\documentclass{amsart}

\usepackage{amsfonts}
\usepackage{amsmath,amsthm,amssymb}
\usepackage{hyperref}
\usepackage{graphicx,color}
\numberwithin{equation}{section}
\usepackage{tikz,ytableau}
\usetikzlibrary{decorations.pathreplacing}

\newtheorem{thm}{Theorem}[section]
\newtheorem{lem}[thm]{Lemma}
\newtheorem{prop}[thm]{Proposition}
\newtheorem{cor}[thm]{Corollary}
\theoremstyle{definition}

\newtheorem{defn}[thm]{Definition}
\newtheorem{conj}[thm]{Conjecture}

\newtheorem{problem}[thm]{Problem}
\newtheorem{remark}[thm]{Remark}

\DeclareMathOperator*{\KK}{\text{\Large \bf K}}

\newcommand\CS{\operatorname{CS}}

\newcommand{\tR}{\widetilde R}

\newcommand\Mot{\operatorname{Motz}}
\newcommand\wt{\operatorname{wt}}

\newcommand\flr[1]{\left\lfloor #1\right\rfloor}

\newcommand\area{\operatorname{area}}
\newcommand\col{\operatorname{col}}
\newcommand\row{\operatorname{row}}

\newcommand\qhyper[5]{
  {}_{#1}\phi_{#2} \left(#3;#4;#5\right)
}

\title{Ratios of Hahn--Exton $q$-Bessel functions and $q$-Lommel polynomials}
\author{Jang Soo Kim and Dennis Stanton}
\thanks{} 
\address{Department of Mathematics,
Sungkyunkwan University, Suwon, South Korea}
\email{jangsookim@skku.edu}
\address{School of Mathematics,
University of Minnesota,
Minneapolis, Minnesota, USA}
\email{stanton@math.umn.edu}
\date{\today}

\begin{document}

\begin{abstract}
  In 1993 Delest and F\'edou showed that a generating function for connected
  skew shapes is given as a ratio $J_{\nu+1}/J_{\nu}$ of the Hahn--Exton
  $q$-Bessel functions when a parameter $\nu$ is zero. They conjectured that
  when $\nu$ is a nonnegative integer the coefficients of the generating
  function are rational functions whose numerator and denominator are
  polynomials in $q$ with nonnegative integer coefficients, which is a
  $q$-analog of Kishore's 1963 result on Bessel functions. The first main result
  of this paper is a proof of the conjecture of Delest and F\'edou. The second
  main result is a refinement of the result of Delest and F\'edou: a generating
  function for connected skew shapes with bounded diagonals is given as a ratio
  of $q$-Lommel polynomials introduced by Koelink and Swarttouw. It is also
  shown that the ratio $J_{\nu+1}/J_{\nu}$ has two different continued fraction
  expressions, which give respectively a generating function for moments of
  orthogonal polynomials of type $R_I$ and a generating function for moments of
  usual orthogonal polynomials. Orthogonal polynomial techniques due to Flajolet
  and Viennot are used.
\end{abstract}

\maketitle

\section{Introduction}

The main objects of study in this paper are ratios of Hahn--Exton $q$-Bessel functions and
$q$-Lommel polynomials, and their combinatorial properties. The \emph{Bessel
  functions} $J_\nu(x)$ are defined by
\[
  J_\nu(z) =\frac{(z/2)^\nu}{\Gamma(\nu+1)} \sum_{n\ge0} \frac{(-z^2/4)^n}{n! (\nu+1)_n}.
\]
It is well known \cite[p.~54]{watson1945} that the Bessel functions generalize
both sine and cosine functions:
\[
  J_{1/2}(z) = \sqrt{\frac{2}{\pi}} \cdot \frac{\sin z}{z^{1/2}},\qquad
  J_{-1/2}(z) = \sqrt{\frac{2}{\pi}} \cdot \frac{\cos z}{z^{1/2}}.
\]
Therefore their ratio gives the tangent function:
\begin{equation}
  \label{eq:tan}
  \frac{J_{1/2}(z)}{J_{-1/2}(z)} = \tan z.
\end{equation}

A motivating example of this paper is the following result of Kishore
\cite{Kishore1964}, which was observed by Lehmer \cite{Lehmer_1944} in 1945. We
consider the ratio $J_{\nu+1}(z)/J_\nu(z)$ is a formal power series in $z$.

\begin{thm}[Kishore, \cite{Kishore1964}]\label{thm:kishore}
We have
\begin{equation}
\frac{J_{\nu+1}(z)}{J_\nu(z)}
= \sum_{n=1}^{\infty} \frac{N_{n,\nu}}{D_{n,\nu}} \left( \frac{z}{2} \right)^{2n-1},
\end{equation}
where 
\[
D_{n,\nu} = \prod_{k=1}^n (k+\nu)^{\lfloor n/k \rfloor},
\]
and
$N_{n,\nu}$ is a polynomial in $\nu$ with nonnegative integer coefficients.
\end{thm}

Throughout this paper we use the standard notation for $q$-series:
\begin{align*}
  (a;q)_n = (1-a)(1-aq)\dots(1-aq^{n-1}),
  \qquad (a;q)_\infty = (1-a)(1-aq)\dots,\\
  \qhyper rs{a_1,\dots,a_r}{b_1,\dots,b_s}{q,z} =
  \sum_{n\ge0} \frac{(a_1,\dots,a_r;q)_n}{(b_1,\dots,b_s;q)_n}
  \left( (-1)^n q^{\binom n2} \right)^{s+1-r} z^n.
\end{align*}
We also define $[\nu]_q = (1-q^\nu)/(1-q)$. Note that if $n$ is a nonnegative
integer, then $[n]_q = 1+q+\cdots+q^{n-1}$. 

In 1993 Delest and F\'edou \cite[Conjecture~9]{Delest1993} conjectured a
$q$-analog of Kishore's theorem using Hahn--Exton $q$-Bessel functions.

\begin{defn}\label{defn:Hahn-Exton}
The \emph{Hahn--Exton $q$-Bessel functions} $J_\nu(z;q)$ are defined by
\[
  J_\nu(z;q) = \frac{(q^{\nu+1};q)_\infty z^\nu}{(q;q)_\infty}
  \qhyper11{0}{q^{\nu+1}}{q,qz^2}.
\]  
\end{defn}

From the definition it follows that
\begin{equation}\label{eq:q->1}
\lim_{q\to 1^-} J_\nu(z(1-q);q) = J_\nu(2z). 
\end{equation}

The first main result of this paper is the following theorem, which is slightly
more general than the conjecture of Delest and F\'edou
\cite[Conjecture~9]{Delest1993}. See Remark~\ref{rema:DF original} for the
original statement of their conjecture.

\begin{thm}
\label{thm:DF}
For any $\nu$ (not a negative integer) we have
\[
\frac{J_{\nu+1}((1-q)z;q^{-1})}{J_\nu((1-q)z;q^{-1})}
= - \sum_{n=1}^{\infty} \frac{N_{n,\nu}(q)}{D_{n,\nu}(q)} q^{(\nu+1)n} z^{2n-1},
\]
where each $N_{n,\nu}(q)$ is a polynomial in $q,q^\nu$ and $[\nu]_q$ with
nonnegative integer coefficients and
\[
D_{n,\nu}(q) = \prod_{k=1}^n [k+\nu]_q^{\lfloor n/k \rfloor}.
\]
\end{thm}

Note that by \eqref{eq:q->1} the $q\to1^+$ limit of Theorem~\ref{thm:DF} with
$z$ replaced by $-q^{-1}z/2$ recovers the Kishore theorem,
Theorem~\ref{thm:kishore}. Our method of proof also works for the Hahn--Exton
$q$-Bessel functions with the usual $q$ base instead of $q^{-1}$ and for
Jackson's $q$-Bessel functions, see Theorems~\ref{thm:DF2} and
\ref{thm:positivity}.

\medskip

Our second main result concerns combinatorial properties of the ratio of
Hahn--Exton $q$-Bessel functions. We first introduce necessary definitions.

\begin{defn}
  A \emph{partition} is a sequence $\sigma=(\sigma_1,\sigma_2,\dots,\sigma_n)$
  of positive integers with
\[
\sigma_1\ge~\sigma_2\ge \dots\ge \sigma_n.
\]
The \emph{Young diagram} of a partition $\sigma$ is a left-justified array of
squares in which the $i$th row has $\sigma_i$ squares.

  If $\sigma$ and $\rho$ are partitions such that the Young diagram of $\rho$ is
  contained in that of $\sigma$, the \emph{skew shape} $\sigma/\rho$ is defined
  to be the set-theoretic difference of their Young diagrams. Two skew shapes
  are considered as the same skew shape if one is obtained from the other by
  translation.

  A skew shape $\sigma/\rho$ is \emph{connected} if for any two squares $u$ and
  $v$ in $\sigma/\rho$ there is a sequence $u_0,u_1,\dots,u_k$ of squares in
  $\sigma/\rho$ such that $u_0=u$, $u_k=v$ and for each $1\le i\le k$ the
  squares $u_i$ and $u_{i-1}$ share an edge.

  Denote by $\CS$ the set of all nonempty connected skew shapes. For
  $\alpha\in\CS$, let $\col(\alpha)$ be the number of nonempty columns in
  $\alpha$, $\row(\alpha)$ the number of nonempty rows in $\alpha$, and
  $\area(\alpha)$ the number of squares in $\alpha$. See Figure~\ref{fig:young}. 
\end{defn}

\begin{figure}
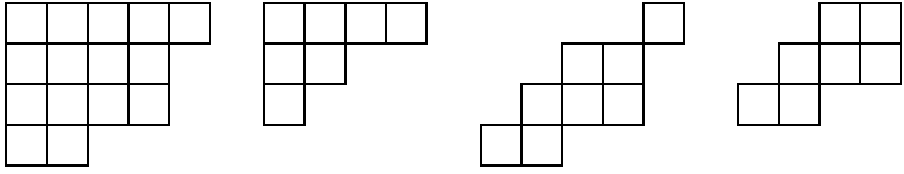

  \centering
\[
\ydiagram{5,4,4,2} \qquad \ydiagram{4,2,1} \qquad \ydiagram{4+1,2+2,1+3,2}
\qquad \ydiagram{2+2,1+3,2}
\]
\caption{From left to right are shown the Young diagrams of the partitions
  $(5,4,4,2)$ and $(4,2,1)$ and the skew shapes $(5,4,4,2)/(4,2,1)$ and
  $(4,4,2)/(2,1)$. The third diagram is not connected and the fourth diagram is
  connected. If $\alpha$ is the fourth diagram, then $\col(\alpha)=4$,
  $\row(\alpha)=3$, and $\area(\alpha)=7$. }
  \label{fig:young}
\end{figure}

Two skew shapes are considered to be equal if one is obtained from the other by
translation. For example, the skew shape $(4,4,2)/(2,1)$ in
Figure~\ref{fig:young} is the same as $(5,4,4,2)/(5,2,1)$ or $(5,5,3)/(3,2,1)$.

Delest and F\'edou \cite{Delest1993} showed that a generating function for
connected skew shapes can be written as a ratio of Hahn--Exton $q$-Bessel
functions. Bousquet-M\'elou and Viennot \cite{Bousquet-Melou_1992} generalized
their result by adding one more parameter. 

\begin{thm}[\cite{Delest1993} for $\nu=0$ and \cite{Bousquet-Melou_1992} for
  general $\nu$]
\label{thm:DF-BM}
We have
\[
  \sum_{\alpha\in \CS} (q^\nu z^2)^{\col(\alpha)} (q^\nu)^{\row(\alpha)} q^{\area(\alpha)} =
 -q^\nu z \frac{J_{\nu+1}(z;q^{-1})}{J_{\nu}(z;q^{-1})}.
\]
\end{thm}

In fact Delest and F\'edou \cite{Delest1993} (for $\nu=0$), and Bousquet-M\'elou
and Viennot \cite{Bousquet-Melou_1992} state their results in the following
equivalent form:
\[
  \sum_{\alpha\in \CS} x^{\col(\alpha)} y^{\row(\alpha)} q^{\area(\alpha)} 
  =  \frac{qxy}{1-qy} \cdot \frac{\qhyper11{0}{q^2y}{q,q^2x}}{\qhyper11{0}{qy}{q,qx}} .
\]
Bousquet-M\'elou and Viennot \cite{Bousquet-Melou_1992} also showed that

\begin{equation}
\label{eq:BM1}
\sum_{\alpha\in \CS} x^{\col(\alpha)} y^{\row(\alpha)} q^{\area(\alpha)}
  = \cfrac{qxy}{
  1-q(x+y) -\cfrac{q^3xy}{
  1-q^2(x+y)- \cfrac{q^5xy}{\cdots}}}.
\end{equation}
We note that in \cite[Corollary 4.6]{Bousquet-Melou_1992} the sequence of the
coefficients of $(x+y)$ in the continued fraction \eqref{eq:BM1} was
inadvertently written $q,q^3,q^5,\dots$, where the correct sequence is
$q,q^2,q^3,\dots$. We also note that there are similar results in
\cite{Barcucci1998}.

\medskip

The second main result of this paper is a finite version of
Theorem~\ref{thm:DF-BM} using $q$-Lommel polynomials. We first recall the connection
between Bessel functions and Lommel polynomials.

The \emph{Lommel polynomials} $R_{n,\nu}(z)$ are polynomials in $z^{-1}$ defined by
$R_{0,\nu}(z) =1$, $R_{1,\nu}(z) = 2\nu/z$, and for $n\ge0$,
\begin{equation}\label{eq:rec R}
  R_{n+1,\nu}(z) = \frac{2(n+\nu)}{z} R_{n,\nu}(z) - R_{n-1,\nu}(z).
\end{equation}
The Bessel functions and the Lommel polynomials are related by the recurrence
\begin{equation}\label{eq:Bessel-Lommel}
  J_{\nu+n}(z) = R_{n,\nu}(z) J_\nu(z) -R_{n-1,\nu+1}(z) J_{\nu-1}(z).
\end{equation}
Hurwitz's theorem \cite[Theorem 6.5.4]{Ismail} says that the Bessel function
$J_\nu(z)$ is obtained as a limit of the Lommel polynomials $R_{n,\nu+1}(z)$:
\begin{equation}
  \label{eq:hurwitz}
  \lim_{n\to\infty} \frac{(z/2)^{n} R_{n,\nu+1}(z)}{\Gamma(n+\nu+1)} =
  (z/2)^{-\nu}J_\nu(z).  
\end{equation}
Thanks to Hurwitz's theorem one can regard Lommel polynomials as a finite analog
of Bessel functions. Note that the ratio of Bessel functions in Kishore's
theorem is the limit of a ratio of Lommel polynomials:
\begin{equation}
  \label{eq:hurwitz2}
  \frac{J_{\nu+1}(z)}{J_\nu(z)} = 
  \lim_{n\to\infty}  \frac{R_{n,\nu+2}(z)}{R_{n+1,\nu+1}(z)}.
\end{equation}

In this paper we study $q$-analogs of these ratios. Koelink and Swarttouw
\cite[(4.18)]{Koelink_1994} introduced the following $q$-Lommel polynomials.

\begin{defn}
  The \emph{$q$-Lommel polynomials} $R_{m,\nu}(z;q)$ associated to the
  Hahn--Exton $q$-Bessel functions $J_{\nu}(z;q)$ are the Laurent polynomials in
  $z$ defined by $R_{-1,\nu}(z;q)=0$, $R_{0,\nu}(z;q)=1$, and for $m\ge0$,
\begin{equation}\label{eq:R}
R_{m+1,\nu}(z;q) = \left( z + z^{-1}(1-q^{\nu+m}) \right) R_{m,\nu}(z;q)
-R_{m-1,\nu}(z;q).
\end{equation}
\end{defn}

In this paper we will consider $J_{\nu}(z;q^{-1})$ and $R_{m,\nu}(z;q^{-1})$ instead
of $J_{\nu}(z;q)$ and $R_{m,\nu}(z;q)$. Koelink and Swarttouw \cite[(4.12),
(4,24)]{Koelink_1994} showed that these $q$-Lommel polynomials satisfy the
following properties analogous to \eqref{eq:Bessel-Lommel} and
\eqref{eq:hurwitz}:
\begin{equation}\label{eq:KS rec2}
  J_{\nu+m}(z;q^{-1}) = R_{m,\nu}(z;q^{-1}) J_{\nu}(z;q^{-1})
  -R_{m-1,\nu+1}(z;q^{-1}) J_{\nu-1}(z;q^{-1}),
\end{equation}
\begin{equation}
  \label{eq:R=J}
  \lim_{m\to\infty} z^m R_{m,\nu}(z;q^{-1})
= \frac{(q^{-1};q^{-1})_\infty z^{1-\nu}}{(z^2;q^{-1})_\infty} J_{\nu-1}(z;q^{-1}).
\end{equation}
By \eqref{eq:R=J} we have a $q$-analog of \eqref{eq:hurwitz2}: 
\begin{equation}\label{eq:R=J2}
  \lim_{n\to\infty}  \frac{R_{n,\nu+2}(z;q^{-1})}{R_{n+1,\nu+1}(z;q^{-1})}
  = \frac{J_{\nu+1}(z;q^{-1})}{J_\nu(z;q^{-1})} .
\end{equation}

Note that
\begin{align}
  \notag
  \frac{J_{\nu+1}(z;q^{-1})}{J_\nu(z;q^{-1})}
  &= \frac{z}{1-q^{-\nu-1}}\cdot
  \frac{\qhyper11{0}{q^{-\nu-2}}{q^{-1},q^{-1}z^2}}
  {\qhyper11{0}{q^{-\nu-1}}{q^{-1},q^{-1}z^2}}\\
  \label{eq:J/J}
  &= \frac{-q^{\nu+1}z}{1-q^{\nu+1}}\cdot
  \frac{\qhyper11{0}{q^{\nu+2}}{q,q^{\nu+2}z^2}}
  {\qhyper11{0}{q^{\nu+1}}{q,q^{\nu+1}z^2}},
\end{align}
where the last equality holds as formal power series in $z$ whose coefficients
are rational functions in $q$ and $q^{\nu}$. Therefore one may consider
\eqref{eq:R=J2} as an equation in the ring of formal power series in $z$,
$q^{-1}$, and $q^{-\nu}$, or in $z$, $q$, and $q^{\nu}$.

Before stating our second main result we need one more definition. 

\begin{defn}
  For a connected skew shape $\alpha$, a \emph{diagonal} is the set of squares
  in row $i$ and column $j$ such that $i-j=k$ for a fixed (not necessary
  positive) integer $k$. Let $\CS^{\le m}$ denote the set of connected skew
  shapes in which every diagonal has at most $m$ squares. See Figure~\ref{fig:diag}.
\end{defn}

\begin{figure}
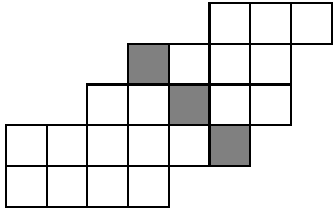

  \centering
  \begin{ytableau}
\none&\none&    \none & \none & \none &&&\\
\none&\none&    \none & *(gray) &&&\\
\none&\none&    {}& &*(gray)&&\\
{}&{}&    {}& &&*(gray)\\
{}&{}&    {}& \\
  \end{ytableau}
  \caption{A diagonal of a skew shape in $\CS^{\le 3}$.}
  \label{fig:diag}
\end{figure}

 Using Flajolet's theory of continued fractions
\cite{Flajolet1980} we show the following finite version of \eqref{eq:BM1}.
\begin{prop}\label{prop:CSm CF}
For any positive integer $m$, we have
\[
  \sum_{\alpha\in \CS^{\le m}} x^{\col(\alpha)} y^{\row(\alpha)} q^{\area(\alpha)}= \cfrac{qxy}{
  1-q(x+y) -\cfrac{q^3xy}{
    1-q^2(x+y)-
    \genfrac{}{}{0pt}{1}{}{\ddots \displaystyle- \cfrac{q^{2m+1}xy}{1-q^{m+1}(x+y)}}
}}.
\]
\end{prop}

The second main result of this paper is the following finite version of Theorem~\ref{thm:DF-BM}. 

\begin{thm}\label{thm:main1}
For any positive integer $m$, we have
\[
  \sum_{\alpha\in \CS^{\le m}} (q^\nu z^2)^{\col(\alpha)} (q^\nu)^{\row(\alpha)} q^{\area(\alpha)}=
-q^\nu z \frac{R_{m,\nu+2}(z;q^{-1})}{R_{m+1,\nu+1}(z;q^{-1})}.
\] 
\end{thm}
If we take the limit $m\to\infty$ in Theorem~\ref{thm:main1}, by
\eqref{eq:R=J2}, we obtain Theorem~\ref{thm:DF-BM}.

\medskip

Note that Theorem~\ref{thm:DF-BM} and \eqref{eq:BM1} give
\begin{equation}
  \label{eq:4}
-q^\nu z \frac{J_{\nu+1}(z;q^{-1})}{J_{\nu}(z;q^{-1})}  
  = \cfrac{q^{2\nu+1}z^2}{
  1-q(q^\nu+q^\nu z^2) -\cfrac{q^{2\nu+3}z^2}{
  1-q^2(q^\nu+q^\nu z^2)- \cfrac{q^{2\nu+5}z^2}{\cdots}}}.
\end{equation}
In Theorem~\ref{thm:ratio of R-1} we show that this continued fraction is a
generating function for moments of orthogonal polynomials of type $R_I$. We will
show in Theorem~\ref{thm:ratio of J-1} that the same ratio
$J_{\nu+1}(z;q^{-1})/J_{\nu}(z;q^{-1})$ can also be expressed as a generating
function for moments of usual orthogonal polynomials.

\medskip

The remainder of this paper is organized as follows. In
Section~\ref{sec:posit-ratio-hahn} we prove the $q$-Kishore theorem
(Theorem~\ref{thm:DF}) using Hahn--Exton $q$-Bessel functions and its analogous
result with the usual $q$ base instead of $q^{-1}$. In
Section~\ref{sec:posit-ratio-jacks} we prove another $q$-Kishore theorem using
Jackson's $q$-Bessel functions. In Section~\ref{sec:orth-polyn-type-1} we review
basic results on orthogonal polynomials of type $R_I$ and continued fractions.
Using these results, in Section~\ref{sec:ratios-q-lommel} we prove
Theorem~\ref{thm:main1} and show that the ratio
$J_{\nu+1}(z;q^{-1})/J_{\nu}(z;q^{-1})$ is a generating function for moments of
orthogonal polynomials of type $R_I$. In Section~\ref{sec:q-norlund-heine} we
show that this ratio can also be written as a generating function for moments of
usual orthogonal polynomials. Finally in Section~\ref{sec:further-study} we
propose some open problems.

\section{$q$-Kishore theorem for Hahn--Exton $q$-Bessel functions}
\label{sec:posit-ratio-hahn}

In this section we prove Theorem~\ref{thm:DF} and an analogous result,
Theorem~\ref{thm:DF2}, with the usual $q$ base instead of $q^{-1}$. 

Kishore \cite{Kishore1964} proved Theorem~\ref{thm:kishore} by
finding a recurrence relation for the coefficient of $z^n$. He did this using
differential equations for Bessel functions. We will use similar methods using
$q$-differential equations to prove Theorem~\ref{thm:DF}.

We now review some facts about $q$-derivatives. Recall that the
\emph{$q$-derivative} $D_q(f(x))$ is defined by
\[
D_q(f(x)) = \frac{f(x)-f(qx)}{x-qx},
\]
and it satisfies the following properties:
\begin{align*}
  D_q(x^n) &= [n]_q x^{n-1},\\
 D_q(f(x)g(x)) &= D_q(f(x)) g(x) + f(qx) D_q(g(x)),\\
 D_q\left( \frac{1}{f(x)} \right) &= \frac{-D_q(f(x))}{f(x)f(qx)}.
\end{align*}

In order to prove Theorem~\ref{thm:DF} we introduce some notation.
By \eqref{eq:J/J}, we have
\[
  \frac{J_{\nu+1}((1-q)z;q^{-1})}{J_\nu((1-q)z;q^{-1})}
  = \frac{-q^{\nu+1}z}{[\nu+1]_q}
  \frac{\qhyper11{0}{q^{\nu+2}}{q,(1-q)^2 z^2q^{\nu+2}}}
  {\qhyper11{0}{q^{\nu+1}}{q,(1-q)^2 z^2q^{\nu+1}}}.
\]
Let $x =qz^2$ and define $\theta_\nu(x)$, $F(x)$, and $\mu_{n}(q)$ by 
\begin{align*}
\theta_\nu (x)&= \qhyper11{0}{q^{\nu+1}}{q,(1-q)^2 xq^\nu}=
\sum_{n=0}^\infty q^{\binom{n}{2}}
\frac{(-1)^n(xq^\nu)^n(1-q)^{2n}}{(q;q)_n(q^{\nu+1};q)_{n}},\\
F(x)&=\frac{\theta_{\nu+1}(x)}{\theta_\nu(x)}=  \sum_{n\ge0} \mu_{n}(q) x^{n},
\end{align*}
so that
\begin{equation}\label{eq:j/j=mu}
\frac{J_{\nu+1}((1-q)z;q^{-1})}{J_\nu((1-q)z;q^{-1})}
= - \frac{q^{\nu+1} z}{[\nu+1]_q}  \sum_{n\ge0} \mu_{n}(q) q^nz^{2n}.
\end{equation}

We give a recurrence relation for $\mu_{n}(q)$.

\begin{prop}
  \label{prop:rec_rel}
We have 
\[
\mu_0(q)=1, \qquad \mu_1(q)= \frac{q^\nu}{[\nu+1]_q[\nu+2]_q},
\]
and for $n\ge 2$,
\[
  q^{-\nu} [\nu+1]_q[\nu+n+1]_q \mu_{n}(q) =
  \sum_{k=2}^{n-1} q^{\nu+k} \mu_{k-1}(q)\mu_{n-k}(q) + (1+q^{\nu+n}) \mu_{n-1}(q).
\]
\end{prop}

\begin{proof}
  The values of $\mu_0$ and $\mu_1$ can be checked easily from the definition.
  We claim that for $n\ge2$,
\begin{equation}
\label{maineq}
\frac{q^{-\nu}[\nu+1]_q}{ x^{\nu+1} }D_q(x^{\nu+1}F(x))=q^{\nu+1}F(x)F(qx)+
(1-q^{\nu+1})F(x)+[\nu+1]_q^2/xq^\nu,
\end{equation}
from which the proposition follows by equating the coefficients of $x^{n-1}$ on both
sides.

Note that
\begin{align*}
D_q\left(x^{\nu+1}\theta_{\nu+1}(x)\right)&= x^\nu\sum_{n=0}^\infty q^{\binom{n}{2}}
\frac{(-1)^n(xq^{\nu+1})^n(1-q)^{2n}}{(q;q)_n(q^{\nu+2};q)_{n-1}} \\
&=x^\nu [\nu+1]_q\theta_\nu(qx)
\end{align*}
and
\begin{align*}
  D_q(\theta_v(x))=
\frac{\theta_\nu(x)-\theta_\nu(qx)}{x-qx}&=
\frac{1}{(1-q)x}\sum_{n=1}^\infty q^{\binom{n}{2}}
\frac{(-1)^n(xq^\nu)^n(1-q)^{2n}}{(q;q)_{n-1}(q^{\nu+1};q)_{n}}\\
&=
\frac{-(1-q)q^\nu}{1-q^{\nu+1}}
\sum_{n=0}^\infty q^{\binom{n}{2}}
\frac{(-1)^n(xq^{\nu+1})^n(1-q)^{2n}}{(q;q)_{n}(q^{\nu+2};q)_{n}}\\
&=\frac{-q^\nu}{[\nu+1]_q}\theta_{\nu+1}(x).
\end{align*}
By the properties of $q$-derivative and the above two identities, we have  
\begin{align*}
&D_q\left(x^{\nu+1}F(x)\right)=D_q\left(\frac{x^{\nu+1}\theta_{\nu+1}(x)}{\theta_\nu(x)}\right)\\
&= D_q\left(\frac{1}{\theta_\nu(x)}\right) 
x^{\nu+1}\theta_{\nu+1}(x)+\frac{1}{\theta_\nu(qx)}D_q(x^{\nu+1}\theta_{\nu+1}(x))\\
&=
\frac{1}{\theta_\nu(qx)\theta_\nu(x)}
\frac{\theta_\nu(qx)-\theta_\nu(x)}{x-qx}x^{\nu+1}\theta_{\nu+1}(x) +
\frac{x^\nu [\nu+1]_q\theta_\nu(qx)}{\theta_\nu(qx)}\\
&=\frac{1}{\theta_\nu(qx)\theta_\nu(x)}
\frac{q^\nu}{[\nu+1]_q}\theta_{\nu+1}(x)x^{\nu+1}\theta_{\nu+1}(x) +
x^\nu [\nu+1]_q\\
&=
\frac{1}{\theta_\nu(qx)\theta_\nu(x)}
\frac{q^\nu}{[\nu+1]_q}\theta_{\nu+1}(x)x^{\nu+1}
\left(q^{\nu+1}\theta_{\nu+1}(qx)+(1-q^{\nu+1})\theta_\nu(qx)\right) +
x^\nu [\nu+1]_q\\
&=
\frac{q^{2\nu+1}}{[\nu+1]_q}x^{\nu+1}F(x)F(qx)+q^\nu(1-q)x^{\nu+1}F(x)+x^\nu [\nu+1]_q.
\end{align*}
Then \eqref{maineq} follows upon multiplying by $q^{-\nu}[\nu+1]_q/x^{\nu+1}$.
\end{proof}

We are now ready to prove Theorem~\ref{thm:DF}, which is stated again below.

\begin{thm}
For any $\nu$ (not a negative integer) we have
\[
\frac{J_{\nu+1}((1-q)z;q^{-1})}{J_\nu((1-q)z;q^{-1})}
= - \sum_{n=1}^{\infty} \frac{N_{n,\nu}(q)}{D_{n,\nu}(q)} q^{(\nu+1)n} z^{2n-1},
\]
where each $N_{n,\nu}(q)$ is a polynomial in $q,q^\nu$ and $[\nu]_q$ with
nonnegative integer coefficients and
\[
D_{n,\nu}(q) = \prod_{k=1}^n [k+\nu]_q^{\lfloor n/k \rfloor}.
\]
\end{thm}
\begin{proof}
  By \eqref{eq:j/j=mu}, Theorem~\ref{thm:DF} can be rewritten as
\[
  \frac{q^{\nu+1} z}{[\nu+1]_q}  \sum_{n\ge0} \mu_{n}(q) q^nz^{2n}=
 \sum_{n=1}^{\infty} \frac{N_{n,\nu}(q)}{D_{n,\nu}(q)} q^{(\nu+1)n} z^{2n-1}
=\sum_{n=0}^{\infty} \frac{N_{n+1,\nu}(q)}{D_{n+1,\nu}(q)} q^{(\nu+1)(n+1)} z^{2n+1},
\]
or equivalently, for each $n\ge0$,
\begin{equation}
  \label{eq:reformulation}
  \frac{\mu_{n}(q)}{q^{n\nu}[\nu+1]_q} = \frac{N_{n+1,\nu}(q)}{D_{n+1,\nu}(q)}.
\end{equation}
Let $d_n =D_{n+1,\nu}(q)= \prod_{k=1}^{n+1}[\nu+k]_q^{\flr{(n+1)/k}}$ and
$\beta_n =N_{n+1,\nu}(q) = \mu_{n}(q) d_n /q^{n\nu} [\nu+1]_q$. Then we must show that
$\beta_n$ is a polynomial in $q$, $q^\nu$, and $[\nu]_q$ with nonnegative
integer coefficients. 

We prove this by induction on $n$. If $n=0$, we have
$\beta_0=1$. If $n=1$, we also have $\beta_1=1$ since $\mu_n(q) =
q^\nu/[\nu+1]_q[\nu+2]_q$ and $d_1 = [\nu+1]_q^2 [\nu+2]_q$. Let $n\ge2$ and
suppose that the claim is true for all $1\le k<n$. By multiplying both sides of
the equation in Proposition~\ref{prop:rec_rel} by
$d_n/q^{(n-1)\nu}[\nu+1]_q^2[\nu+n+1]_q$ we obtain
\[
  \beta_n = \sum_{k=2}^{n-1} q^{\nu+k}\beta_{k-1}\beta_{n-k} \frac{d_n}{d_{k-1} d_{n-k} [\nu+n+1]_q}
  +(1+q^{\nu+n})  \beta_{n-1} \frac{d_n}{d_{n-1} [\nu+n+1]_q}.
\]
It is straightforward to check that 
for $2\le k\le n-1$,
\[
\frac{d_n}{d_{k-1} d_{n-k} [\nu+n+1]_q},\qquad
\frac{d_n}{d_{n-1} [\nu+n+1]_q}
\]
  are polynomials in
$q$, $q^\nu$, and $[\nu]_q$ with nonnegative integer coefficients
using the fact $[\nu+j]_q  = [\nu]_q+q^\nu[j]_q$. Therefore
$\beta_n$ is also a polynomial with nonnegative coefficients, which completes
the proof by induction.
\end{proof}

\begin{remark}\label{rema:DF original}
  Delest and F\'edou \cite{Delest1993} in fact considered the following modified
  $q$-Bessel functions. For a nonnegative integer $\nu$, let
\[
  T_\nu(x;q) = \sum_{n\ge0} \frac{(-1)^n q^{\binom{n+\nu}2} x^{n+\nu}}{[n]_q![n+\nu]_q!}
  =\frac{q^{\binom \nu2} ((1-q)x)^\nu}{(q;q)_\nu}
  \qhyper11{0}{q^{\nu+1}}{q,(1-q)^2q^\nu x},
\]
where $[n]_q! = [1]_q[2]_q\cdots [n]_q$. Then the original form of the
conjecture in \cite[Conjecture 9]{Delest1993} states that
\[
\frac{T_{\nu+1}(x;q)}{T_\nu(x;q)} = \sum_{n=1}^{\infty} \frac{N_{n,\nu}(q)}{D_{n,\nu}(q)} x^n,
\]
where $N_{n,\nu}(q)$ and $D_{n,\nu}(q)$ are as in Theorem~\ref{thm:DF}. It is
straightforward to check that
\[
  \frac{T_{\nu+1}(x;q)}{T_\nu(x;q)} = -x^{1/2}q^{-1/2}
  \frac{ J_{\nu+1}((1-q)x^{1/2}q^{-1/2};q^{-1})}
  { J_\nu((1-q)x^{1/2}q^{-1/2};q^{-1})},
\]
which shows that their conjecture follows from Theorem~\ref{thm:DF}.
\end{remark}

For the rest of this section as corollaries of Theorem~\ref{thm:DF} we give two
more $q$-analogs of Kishore's theorem using the usual $q$ base instead of
$q^{-1}$. Replacing $q$ by $q^{-1}$ in \eqref{eq:j/j=mu}, we obtain
\begin{equation}\label{eq:j/j=rho}
\frac{J_{\nu+1}((1-q^{-1})z;q)}{J_\nu((1-q^{-1})z;q)}
= - \frac{q^{-\nu-1} z}{[\nu+1]_{q^{-1}}}  \sum_{n\ge0} \mu_{n}(q^{-1}) q^{-n}z^{2n}.
\end{equation}
If we replace $z$ by $-qz$ in \eqref{eq:j/j=rho} we obtain
\begin{equation}\label{eq:j/j=rho2}
\frac{J_{\nu+1}((1-q)z;q)}{J_\nu((1-q)z;q)}
= \frac{1}{[\nu+1]_{q}}  \sum_{n\ge0} \mu_{n}(q^{-1}) q^{n}z^{2n+1}.
\end{equation}

Therefore by replacing $q$ by $q^{-1}$ in Proposition~\ref{prop:rec_rel} and
using the fact $[\nu+1]_{q^{-1}}= q^{-\nu}[\nu+1]_q$ we obtain the following
recurrence relation for $\mu_{n}(q^{-1})$.

\begin{prop}
  \label{prop:rec_rel2}
We have 
\[
\mu_0(q^{-1})=1, \qquad \mu_1(q^{-1})= \frac{q^{\nu+1}}{[\nu+1]_q[\nu+2]_q},
\]
and for $n\ge 2$,
\[
  [\nu+1]_q[\nu+n+1]_q \mu_{n}(q^{-1}) =
  \sum_{k=2}^{n-1} q^{n-k} \mu_{k-1}(q^{-1})\mu_{n-k}(q^{-1})
  + (1+q^{\nu+n}) \mu_{n-1}(q^{-1}).
\]
\end{prop}

Using the recurrences in Propositions~\ref{prop:rec_rel} and \ref{prop:rec_rel2}
the following relation between $\mu_n(q^{-1})$ and $\mu_n(q)$ is easily shown by
induction.

\begin{prop}\label{prop:rho=mu}
  For $n\ge1$, 
\[
\mu_n(q^{-1}) = q^{1-(n-1)\nu} \mu_n(q). 
\] 
\end{prop}

Now we can prove a second $q$-analog of Kishore's theorem.

\begin{thm}
\label{thm:DF2}
For any $\nu$ (not a negative integer) we have
\[
\frac{J_{\nu+1}((1-q)z;q)}{J_\nu((1-q)z;q)}
= (1-q)z + \sum_{n\ge1}\frac{N_{n,\nu}(q)}{D_{n,\nu}(q)} q^{n+\nu} z^{2n-1},
\]
where $N_{n,\nu}(q)$ and $D_{n,\nu}(q)$ are the same polynomials given in
Theorem~\ref{thm:DF}.
\end{thm}
\begin{proof}
  By Proposition~\ref{prop:rho=mu} and \eqref{eq:reformulation}, we have
  \begin{align*}
\frac{J_{\nu+1}((1-q)z;q)}{J_\nu((1-q)z;q)}
&= \frac{1}{[\nu+1]_{q}} \left( z+ \sum_{n\ge1} \mu_{n}(q^{-1}) q^{n}z^{2n+1} \right)\\
&= \frac{1}{[\nu+1]_{q}} \left( z+ \sum_{n\ge1} q^{n+1-(n-1)\nu}\mu_{n}(q) z^{2n+1} \right)\\
&= \frac{1}{[\nu+1]_{q}} \left( z-q^{1+\nu}z+ \sum_{n\ge0} q^{n+1-(n-1)\nu}\mu_{n}(q) z^{2n+1} \right)\\
&= (1-q)z + \sum_{n\ge0}\frac{N_{n+1,\nu}(q)}{D_{n+1,\nu}(q)} q^{n+1+\nu} z^{2n+1} \\
&= (1-q)z + \sum_{n\ge1}\frac{N_{n,\nu}(q)}{D_{n,\nu}(q)} q^{n+\nu} z^{2n-1},
  \end{align*}
as desired.
\end{proof}

Comparing Theorems~\ref{thm:DF} and \ref{thm:DF2} we obtain the following
corollary. 

\begin{cor}\label{cor:cor 2.7}
For any $\nu$ (not a negative integer) we have
\[
\frac{J_{\nu+1}(-(1-q)q^{-\nu/2}z;q^{-1})}{J_\nu(-(1-q)q^{-\nu/2}z;q^{-1})}
=q^{-\nu/2} \frac{J_{\nu+1}((1-q)z;q)}{J_\nu((1-q)z;q)}
-(1-q)q^{-\nu/2}z.
\] 
\end{cor}
\begin{proof}
  By Theorems~\ref{thm:DF} and \ref{thm:DF2} we have
  \begin{align*}
    \frac{J_{\nu+1}(-(1-q)q^{-\nu/2}z;q^{-1})}{J_\nu(-(1-q)q^{-\nu/2}z;q^{-1})}
    &= \sum_{n\ge1}\frac{N_{n,\nu}(q)}{D_{n,\nu}(q)} q^{(\nu+1)n} q^{-\frac{\nu}{2}(2n-1)}z^{2n-1}\\
    &= q^{-\nu/2}\sum_{n\ge1}\frac{N_{n,\nu}(q)}{D_{n,\nu}(q)} q^{\nu+n} z^{2n-1}\\
    &= q^{-\nu/2}\left(\frac{J_{\nu+1}((1-q)z;q)}{J_\nu((1-q)z;q)} -(1-q)z \right),
  \end{align*}
as desired.
\end{proof}

We note that Corollary~\ref{cor:cor 2.7} can also be proved directly
by finding the coefficients of $z^{2k+1}$ on both sides.

Combining Definition~\ref{defn:Hahn-Exton}
and \eqref{eq:J/J} gives
\begin{equation}\label{eq:q and 1/q}
\frac{J_{\nu+1}((1-q)z;q^{-1})}{J_\nu((1-q)z;q^{-1})}
= -q^{1/2} \frac{J_{\nu+1}(q^{(\nu+1)/2}(1-q)z;q)}{J_\nu(q^{\nu/2}(1-q)z;q)}.
\end{equation} 

By Theorem~\ref{thm:DF} and \eqref{eq:q and 1/q} we
obtain a third $q$-analog of Kishore's theorem.

\begin{thm}
\label{thm:DF3}
For any $\nu$ (not a negative integer) we have
\[
\frac{J_{\nu+1}((1-q)q^{1/2}z;q)}{J_\nu((1-q)z;q)}
=  q^{(\nu-1)/2} \sum_{n\ge1}\frac{N_{n,\nu}(q)}{D_{n,\nu}(q)} q^n z^{2n-1},
\]
where $N_{n,\nu}(q)$ and $D_{n,\nu}(q)$ are the same polynomials given in
Theorem~\ref{thm:DF}.
\end{thm}

\section{$q$-Kishore theorem for Jackson's $q$-Bessel functions}
\label{sec:posit-ratio-jacks}

In this section we show another $q$-analog of Kishore's theorem using a
different $q$-Bessel function, called Jackson's
$q$-Bessel function. The arguments are similar to those in the previous section.

\begin{defn}
The \emph{Jackson $q$-Bessel functions} $J^{(k)}_\nu(z;q)$, $k=1,2$, are defined
by
\begin{align}
\label{eq:J1}  J^{(1)}_\nu(z;q)
  &= \frac{(q^{\nu+1};q)_\infty}{(q;q)_\infty} (z/2)^\nu \qhyper21{0,0}{q^{\nu+1}}{q,-z^2/4},\\
\label{eq:J2}  J^{(2)}_\nu(z;q)
  &= \frac{(q^{\nu+1};q)_\infty}{(q;q)_\infty} (z/2)^\nu \qhyper01{-}{q^{\nu+1}}{q,-q^{\nu+1}z^2/4}.
\end{align}
\end{defn}
From the definitions it follows easily that, for $k=1,2$,
\[
\lim_{q\to 1^-} J^{(k)}_\nu(x(1-q);q) = J_\nu(x). 
\]
There is a simple connection between $J^{(1)}_\nu(z;q)$ and $J^{(2)}_\nu(z;q)$, see
\cite[Theorem~14.1.3]{Ismail}:
\begin{equation}
  \label{eq:J1=J2}
(-z^2/4;q)_\infty  J^{(1)}_\nu(z;q) =  J^{(2)}_\nu(z;q).
\end{equation}
By \eqref{eq:J1=J2}, in order to compute the ratio
$J^{(k)}_{\nu+1}(z;q)/J^{(k)}_\nu(z;q)$ for $k=1,2$, it suffices to consider $k=1$.

\begin{thm}\label{thm:positivity}
For any $\nu$ (not a negative integer) we have
\begin{equation}\label{eq:J1/J1}
  \frac{J^{(1)}_{\nu+1}((1-q)z;q)}{J^{(1)}_\nu((1-q)z;q)} = \sum_{n=1}^{\infty}
  \frac{N^{(1)}_{n,\nu}(q)}{D_{n,\nu}(q)} q^{(\nu+1)(n-1)} \left( \frac{z}{2}
  \right)^{2n-1},
\end{equation}
where each $N^{(1)}_{n,\nu}(q)$ is a polynomial in $q,q^\nu$ and $[\nu]_q$ with
nonnegative integer coefficients and
\[
D_{n,\nu}(q) = \prod_{k=1}^n [k+\nu]_q^{\lfloor n/k \rfloor}.
\]
\end{thm}

As in the previous section we introduce some notation. 
 By definition we have
\[
  \frac{J^{(1)}_{\nu+1}((1-q)z;q)}{J^{(1)}_\nu((1-q)z;q)}
  =\frac{z/2}{[\nu+1]_q}
  \frac{\qhyper21{0,0}{q^{\nu+2}}{q,-(1-q)^2z^2/4}}{\qhyper21{0,0}{q^{\nu+1}}{q,-(1-q)^2z^2/4}}.
\]
Let $x=z^2/4$ and define $\theta^{(1)}_\nu(x)$, $F^{(1)}(x)$, and $\mu^{(1)}_{n}(q)$ by 
 \begin{align*}
   \theta^{(1)}_\nu(x)&=\qhyper21{0,0}{q^{\nu+1}}{q,-(1-q)^2x},\\
F^{(1)}(x)&=\frac{\theta^{(1)}_{\nu+1}(x)}{\theta^{(1)}_\nu(x)}
=\sum_{n=0}^\infty \mu^{(1)}_{n}(q)x^n,
 \end{align*}
so that
\begin{equation}\label{eq:j1/j1=mu}
  \frac{J^{(1)}_{\nu+1}((1-q)z;q)}{J^{(1)}_\nu((1-q)z;q)}
  =\frac{z/2}{[\nu+1]_q} \sum_{n\ge0}\mu^{(1)}_{n}(q)z^{2n}.
\end{equation}

We first show a recurrence relation for $\mu_{n}$. 

\begin{prop}\label{prop:j1 rec}
We have $\mu^{(1)}_0(q)=1$ and for $n\ge 1$,
$$
q^{-\nu-1}[\nu+1]_q[\nu+n+1]_q\mu^{(1)}_{n}(q)=\sum_{k=0}^{n-1} q^{k}\mu^{(1)}_{k}(q)\mu^{(1)}_{n-k-1}(q).
$$
\end{prop}
\begin{proof} 
We claim that 
$$
\frac{[\nu+1]_q}{(qx)^{\nu+1}} D_q(x^{\nu+1} F(x))=F(x)F(qx)+
\frac{[\nu+1]^2}{q^{\nu+1}x},
$$
from which the proposition follows by equating the coefficients of
$x^{n-1}$ on both sides.

 The $q$-product rule implies
 \begin{align}
   \notag
   \frac{[\nu+1]_q}{(qx)^{\nu+1}} D_q(x^{\nu+1} F(x))
   &=\frac{[\nu+1]_q}{(qx)^{\nu+1}} D_q\left(\frac{x^{\nu+1} \theta_{\nu+1}(x)}{\theta_{\nu}(x)}\right)\\
\label{needeqn}
   &=\frac{[\nu+1]_q}{(qx)^{\nu+1}} \frac{D_q(x^{\nu+1}\theta_{\nu+1}(x))}{\theta_{\nu}(x)}+
     [\nu+1]_q\theta_{\nu+1}(qx)D_q\left(\frac{1}{\theta_{\nu}(x)}\right).
 \end{align}
Since
$$
\frac{D_q(x^{\nu+1}\theta_{\nu+1}(x))}{x^{\nu+1}}=\sum_{n=0}^\infty
\frac{(-1)^n(1-q)^{2n}x^{n-1}}{(q;q)_n (q^{\nu+2};q)_n}\frac{1-q^{\nu+n+1}}{1-q}
=[\nu+1]_q \theta_{\nu}(x)/x
$$
and
$$
\begin{aligned}
D_q(1/\theta_{\nu}(x))&=(1/\theta_{\nu}(x)-1/\theta_{\nu}(qx))/(1-q)x\\
&= \frac{\theta_{\nu}(qx)-\theta_{\nu}(x)}{(1-q)x\theta_{\nu}(x)\theta_{\nu}(qx)}\\
&=- \frac{1}{\theta_{\nu}(x)\theta_{\nu}(qx)}
 \sum_{n=1}^\infty \frac{(-1)^nx^{n-1}(1-q)^{2n}}{(q;q)_n (q^{\nu+1};q)_n}\frac{1-q^{n}}{1-q}\\
&= \frac{1}{\theta_{\nu}(x)\theta_{\nu}(qx)}
\frac{1}{[\nu+1]_q}\sum_{n=0}^\infty \frac{x^{n}(1-q)^{2n}}{(q;q)_n (q^{\nu+2};q)_n}\\
&= \frac{1}{[\nu+1]_q} \frac{\theta_{\nu+1}(x)}{\theta_{\nu}(x)\theta_{\nu}(xq)}
\end{aligned}
$$
we see that \eqref{needeqn} becomes
\begin{align*}
  \frac{[\nu+1]_q}{(qx)^{\nu+1}} D_q(x^{\nu+1} F(x))
  &=\frac{[\nu+1]_q^2}{q^{\nu+1}x} +\frac{\theta_{\nu+1}(x)\theta_{\nu+1}(qx)}{\theta_{\nu}(x)\theta_{\nu}(xq)}\\
&=\frac{[\nu+1]_q^2}{q^{\nu+1}x} +F(x)F(qx)
\end{align*}
as required. 
\end{proof}

\begin{proof}[Proof of Theorem~\ref{thm:positivity}]
By \eqref{eq:j1/j1=mu} we have
\[
\frac{1}{[\nu+1]_q} \sum_{n\ge0}\mu^{(1)}_{n}(q)\left( \frac{z}{2}
  \right)^{2n} =\sum_{n=1}^{\infty}
  \frac{N^{(1)}_{n,\nu}(q)}{D_{n,\nu}(q)} q^{(\nu+1)(n-1)} \left( \frac{z}{2}
  \right)^{2n-2}.
\]
Let $d_n=D_{n+1,\nu}(q)= \prod_{k=1}^{n+1}[\nu+k]_q^{\flr{(n+1)/k}}$ and $\beta_n= N^{(1)}_{n+1,\nu}(q)$. Then we need to
show that
\[
\beta_n = \frac{\mu^{(1)}_{n}(q)d_n}{q^{(\nu+1)n}[\nu+1]_q}
\]
is a polynomial in $q,q^\nu$ and $[\nu]_q$ with nonnegative integer coefficients.
We prove this by induction on $n$. It is true for $n=0$ since $\beta_0=1$. 

For the inductive step let $n\ge1$ and suppose that $\beta_k$ is a polynomial in
$q,q^\nu$, and $[\nu]_q$ for all $0\le k<n$. Multiplying both sides of the
equation in Proposition~\ref{prop:j1 rec} by
$d_n/q^{(\nu+1)(n-1)}[\nu+1]_q^2[\nu+n+1]_q$, we obtain
\begin{equation}
  \label{eq:3}
\beta_n = \sum_{k=0}^{n-1} \frac{q^{k}d_n}{d_{k}d_{n-1-k}[\nu+n+1]_q} \beta_k\beta_{n-1-k}.
\end{equation}
It is easy to check that for $0\le k\le n-1$, 
\[
\frac{d_n}{d_{k}d_{n-1-k}[\nu+n+1]_q}
\]
is a polynomial in $q,q^\nu$ and $[\nu]_q$ with nonnegative integer
coefficients. Then by induction hypothesis \eqref{eq:3} shows that $\beta_n$ is
also such a polynomial, completing the proof.
\end{proof}

As in the previous section we can also obtain a similar result using the base
$q^{-1}$. Replacing $q$ by $q^{-1}$ in \eqref{eq:j1/j1=mu}, we obtain
\begin{equation}\label{eq:j1/j1=rho}
\frac{J^{(1)}_{\nu+1}((1-q^{-1})z;q^{-1})}{J^{(1)}_\nu((1-q^{-1})z;q^{-1})}
 =\frac{q^\nu z/2}{[\nu+1]_q} \sum_{n\ge0}\mu^{(1)}_{n}(q^{-1})z^{2n}.
\end{equation}

Using Proposition~\ref{prop:j1 rec} it is easy to show that for $n\ge0$,
\begin{equation}\label{eq:mu1=mu1}
  \mu_n^{(1)}(q) = q^{n}\mu^{(1)}_n(q^{-1}) .
\end{equation}
Replacing $z$ by $-qz$ in \eqref{eq:j1/j1=rho} and using \eqref{eq:mu1=mu1} gives
\begin{equation}\label{eq:j1/j1=rho2}
\frac{J^{(1)}_{\nu+1}((1-q)z;q^{-1})}{J^{(1)}_\nu((1-q)z;q^{-1})}
=- \frac{q^{\nu+1} z/2}{[\nu+1]_{q}} \sum_{n\ge0} \mu^{(1)}_{n}(q) q^{n}z^{2n}.
\end{equation}

By \eqref{eq:j1/j1=mu} and \eqref{eq:j1/j1=rho2},
\begin{equation}
  \label{eq:1}
\frac{J^{(1)}_{\nu+1}((1-q)z;q^{-1})}{J^{(1)}_\nu((1-q)z;q^{-1})}
= -q^{\nu+\frac{1}{2}}
\frac{J^{(1)}_{\nu+1}((1-q)q^{1/2}z;q)}{J^{(1)}_\nu((1-q)q^{1/2}z;q)}.
\end{equation}

Therefore by Theorem~\ref{thm:positivity} and \eqref{eq:1} we obtain an
analogous result of Theorem~\ref{thm:positivity}.

\begin{thm}\label{thm:positivity2}
For any $\nu$ (not a negative integer) we have
\begin{equation}\label{eq:J1/J12}
  \frac{J^{(1)}_{\nu+1}((1-q)z;q^{-1})}{J^{(1)}_\nu((1-q)z;q^{-1})}
  = - \sum_{n=1}^{\infty}
  \frac{N^{(1)}_{n,\nu}(q)}{D_{n,\nu}(q)} q^{\nu n - 1} \left( \frac{z}{2}
  \right)^{2n-1},
\end{equation}
where $N^{(1)}_{n,\nu}(q)$ and $D_{n,\nu}(q)$ are the same polynomials given in
Theorem~\ref{thm:positivity}.
\end{thm}

We note that Li \cite[Proposition~4.5]{li18:carlit_scovil_vaugh} gave a
combinatorial meaning to the coefficients of the power series expansion of
$1/J^{(1)}_0(z;q)$.

\section{Orthogonal polynomials of type $R_I$ and continued fractions}
\label{sec:orth-polyn-type-1}

In this section we review basic results in \cite{IsmailMasson} and
\cite{kimstanton:R1} on orthogonal polynomials of type $R_I$
and continued fractions, which will be used in the next section. We begin by
introducing some notation for continued fractions.

\begin{defn}
For sequences $a_i$ and $b_i$, let
\[
 \KK_{i=0}^m \left( \frac{a_i}{b_i} \right) =
  \cfrac{a_0}{
  b_0 + \cfrac{a_1}{
  b_1 + \genfrac{}{}{0pt}{0}{}{\displaystyle\ddots + \cfrac{a_m}{b_m}}}}, \qquad
 \KK_{i=0}^\infty \left( \frac{a_i}{b_i} \right) =
  \cfrac{a_0}{
  b_0 + \cfrac{a_1}{
  b_1 + \genfrac{}{}{0pt}{0}{}{\ddots }}}.
\]
\end{defn}

The following lemma will be used later. 

\begin{lem}
  \label{lem:K=K}
For any sequences $\{a_i:0\le i\le m\}$, $\{b_i:0\le i\le m\}$, and $\{c_i:-1\le i\le m\}$, we have
\[
 \KK_{i=0}^m \left( \frac{a_i}{b_i} \right) =
\frac{1}{c_{-1}} \KK_{i=0}^m \left( \frac{a_ic_{i-1}c_i}{b_ic_i} \right).
\]
\end{lem}
\begin{proof}
By multiplying $c_i$ to by the numerator and denominator of the $i$th fraction,
we obtain
\[
  \cfrac{a_0}{
  b_0 + \cfrac{a_1}{
  b_1 \displaystyle+ \genfrac{}{}{0pt}{0}{}{\ddots + \cfrac{a_m}{b_m}}}}
=  \cfrac{a_0c_0}{
  b_0c_0 + \cfrac{a_1c_0c_1}{
  b_1c_1 + \genfrac{}{}{0pt}{0}{}{\ddots + \cfrac{a_mc_{m-1}c_m}{b_mc_m}}}},
\]
which is equivalent to the equation in the lemma.
\end{proof}

Ismail and Masson \cite{IsmailMasson} introduced orthogonal polynomials of type
$R_I$ generalizing the usual orthogonal polynomials. 

\begin{defn}
  A family of polynomials $p_n(x)$, $n\ge0$, is called \emph{(monic) orthogonal
    polynomials of type $R_I$} if they satisfy the three term recurrence
  relation: $p_{-1}(x)=0$, $p_0(x)=1$, and for $n\ge0$,
\[
p_{n+1}(x) = (x-b_n) p_n(x )- (a_nx+\lambda_n) p_{n-1}(x),
\]
for some sequences $b=\{b_k\}_{k\ge0}$, $a=\{a_k\}_{k\ge0}$, and
$\lambda=\{\lambda_k\}_{k\ge0}$. In this case we say that $p_n(x)$ are the
orthogonal polynomials of type $R_I$ determined by the sequences $b,a,\lambda$.
\end{defn}

The usual (monic) orthogonal polynomials are orthogonal polynomials of type
$R_I$ determined by sequences $b,a,\lambda$ with $a_n=0$ for all $n$. If
$\lambda_n=0$ for all $n$, then the orthogonal polynomials of type $R_I$ becomes
orthogonal Laurent polynomials, see \cite{jones1982survey, Kamioka2014} for more
details on orthogonal Laurent polynomials. In the next section we will see that
the $q$-Lommel polynomials $\{R_{n,\nu}(x)\}_{n\ge0}$ are orthogonal polynomials
type $R_I$ with $\lambda_n=0$.

For any sequence $a=\{a_k\}_{k\ge0}$, let $\delta a=\{\delta a_k\}_{k\ge0}$
denote the sequence obtained by shifting the index up by $1$, i.e., $\delta a_k
= a_{k+1}$. Given orthogonal polynomials $p_n(x)$ of type $R_I$ we consider two
other sequences $\{\delta p_n(x)\}$ and $\{p^*_n(x)\}$ of polynomials associated
to it as follows.

\begin{defn}
  Let $\{p_n(x)\}$ be the orthogonal polynomials of type $R_I$ determined by the
  sequences $b,a,\lambda$. We define $\{\delta p_n(x)\}_{n\ge0}$ to be the
  orthogonal polynomials of type $R_I$ determined by $\delta b, \delta a, \delta
  \lambda$, and define
  \[
    p^*_n(x) = x^n p_n(1/x).
  \]
\end{defn}

We will see that the quantities $\mu_n^{\le m}(b,a,\lambda)$ and
$\mu_n(b,a,\lambda)$ in the following definition are closely related to
orthogonal polynomials of type $R_I$.

\begin{defn}\label{defn:mu}
  Let $b=\{b_k\}_{k\ge0}$, $a=\{a_k\}_{k\ge0}$, and
  $\lambda=\{\lambda_k\}_{k\ge0}$ be sequences. For $m\ge0$, we define
  $\mu_n^{\le m}(b,a,\lambda)$ and $\mu_n(b,a,\lambda)$ by
\begin{align*}
  \sum_{n\ge0} \mu_n^{\le m}(b,a,\lambda) z^n
  &= \cfrac{1}{
    1-b_0z -\cfrac{a_1z+\lambda_1 z^2}{
      1-b_1z-\cfrac{a_2z+\lambda_2 z^2}{
        1-b_2z- \genfrac{}{}{0pt}{1}{}{\displaystyle\ddots -
          \cfrac{a_mz+\lambda_m z^2}{1-b_m z}}}}}\\
  &= \frac{1}{-a_0z-\lambda_0z^2} \KK_{i=0}^m
    \left( \frac{-a_iz-\lambda_i z^2}{1-b_iz} \right) ,\\
  \sum_{n\ge0} \mu_n(b,a,\lambda) z^n
  &= \cfrac{1}{
    1-b_0z -\cfrac{a_1z+\lambda_1 z^2}{
      1-b_1z-\cfrac{a_2z+\lambda_2 z^2}{
        1-b_2z- \genfrac{}{}{0pt}{1}{}{\ddots}}}}\\
  &= \frac{1}{-a_0z-\lambda_0z^2} \KK_{i=0}^\infty
    \left( \frac{-a_iz-\lambda_i z^2}{1-b_iz} \right) .
\end{align*}
\end{defn}

The continued fractions in the above definition are formal power series in $z$
whose coefficients are polynomials in the elements of $b,a$, and $\lambda$. The
infinite continued fraction always converges in this formal power series ring.

The following theorem can be proved using standard techniques in continued
fractions, see \cite{kimstanton:R1} for a proof. The quantities
$\mu_n(b,a,\lambda)$ in this theorem are called the \emph{moments} of the
orthogonal polynomials of type $R_I$. This result for the usual orthogonal
polynomials, the case $a_n=0$, is well known and Viennot \cite{ViennotLN}
developed Flajolet's \cite{Flajolet1980} combinatorial theory in this case using
Motzkin paths. Kamioka \cite[Proof of Lemma~3.3]{Kamioka2014} gave a
combinatorial model for the case $\lambda_n=0$ using Schr\"oder paths.
Flajolet's theory can also be applied to the general case. A combinatorial model
for the general case is given in \cite{kimstanton:R1} using certain lattice
paths generalizing both Motzkin paths and Schr\"oder paths.

\begin{prop}\label{prop:dp/p=mu}
  Let $p_n(x)$ be the type $R_I$ orthogonal polynomials determined by the
  sequences $b,a$, and $\lambda$. Then for $m\ge0$,
\[
  \frac{\delta p^*_{m}(x)}{p^*_{m+1}(x)}
  =\sum_{n \ge0} \mu^{\le m}_n(b,a,\lambda) x^n,
\]
\[
  \lim_{m\to\infty} \frac{\delta p^*_{m}(x)}{p^*_{m+1}(x)}
  =\sum_{n \ge0} \mu_n(b,a,\lambda) x^n.
\]
\end{prop}

\section{Ratios of $q$-Lommel polynomials}
\label{sec:ratios-q-lommel}

In this section we prove Theorem~\ref{thm:main1}, which gives a combinatorial
interpretation for the ratio $R_{m,\nu+2}(z;q^{-1})/R_{m+1,\nu+1}(z;q^{-1})$ of
$q$-Lommel polynomials. We also show that the ratio
$J_{\nu+1}(z;q^{-1})/J_{\nu}(z;q^{-1})$ is a generating function for moments of
orthogonal polynomials of type $R_I$.

We first show that the $q$-Lommel polynomials
$R_{m,\nu}(x;q)$ give rise to orthogonal polynomials of type $R_I$. Then using
Proposition~\ref{prop:dp/p=mu} we interpret the ratio of $q$-Lommel polynomials as a
continued fraction. We show that the continued fraction is a generating function
for Motzkin paths of bounded height. Finally, we use a bijection between Motzkin
paths and connection skew shapes to obtain the theorem.

Define \emph{modified $q$-Lommel polynomials} $\tR_{m,\nu}(x;q)$ by
\[
\tR_{m,\nu}(x;q) = \frac{x^{m/2}}{(q^\nu;q)_m} R_{m,\nu}(x^{-1/2};q).
\]
 Then, by \eqref{eq:R}, we have
$\tR_{0,\nu}(x;q)=1$, $\tR_{-1,\nu}(x;q)=0$, and for $m\ge0$,
\[
  \tR_{m+1,\nu}(x;q) = \left(x+\frac{1}{1-q^{\nu+m}}\right)
  \tR_{m,\nu}(x;q) - \frac{x}{(1-q^{\nu+m-1})(1-q^{\nu+m})} \tR_{m-1,\nu}(x;q),
\]
which is equivalent to the recurrence in \cite[(4.20)]{Koelink_1994}. By
replacing $q$ by $q^{-1}$ in the above recurrence we obtain
\begin{equation}\label{eq:RR tR}
  \tR_{m+1,\nu}(x;q^{-1}) = \left(x-\frac{q^{\nu+m}}{1-q^{\nu+m}}\right)
  \tR_{m,\nu}(x;q^{-1}) - \frac{q^{2\nu+2m-1}x}{(1-q^{\nu+m-1})(1-q^{\nu+m})}
  \tR_{m-1,\nu}(x;q^{-1}).
\end{equation}

\begin{thm}\label{thm:ratio of R-1}
  Let $b=\{b_n\}_{n\ge0}$, $a=\{a_n\}_{n\ge0}$, and
  $\lambda=\{\lambda_n\}_{n\ge0}$, where
  \[
b_n=\frac{q^{\nu+n+1}}{1-q^{\nu+n+1}},\qquad a_n
  =\frac{q^{2\nu+2n+1}}{(1-q^{\nu+n})(1-q^{\nu+n+1})} ,\qquad \lambda_n =0.
  \]
  Then
  \begin{align}
    \label{eq:ratio of R-1}
    \frac{R_{m,\nu+2}(z;q^{-1})}{R_{m+1,\nu+1}(z;q^{-1})} 
    &=  \frac{q^{\nu+1}z}{q^{\nu+1}-1} \sum_{n\ge0}\mu_n^{\le m}(b,a,\lambda) z^{2n},\\
    \label{eq:ratio of J-1}
    \frac{J_{\nu+1}(z;q^{-1})}{J_{\nu}(z;q^{-1})}
    &= \frac{q^{\nu+1}z}{q^{\nu+1}-1} \sum_{n\ge0}\mu_n(b,a,\lambda) z^{2n}.
  \end{align}
\end{thm}
\begin{proof}
  Let $p_n(x)=\tR_{n,\nu+1}(x;q^{-1})$. Then by \eqref{eq:RR tR}, $p_n(x)$ are
  orthogonal polynomials of type $R_I$ determined by $b,a$, and $\lambda$, and
  we have
  \begin{align*}
    p^*_m(x) &= x^m p_m(1/x) = x^m\tR_{m,\nu+1}(1/x;q^{-1})=
    \frac{x^{m/2}}{(q^{-\nu-1};q^{-1})_m} R_{m,\nu+1}(x^{1/2};q^{-1}) ,\\
\delta p^*_m(x) &= \frac{x^{m/2}}{(q^{-\nu-2};q^{-1})_m} R_{m,\nu+2}(x^{1/2};q^{-1}).
\end{align*}
By Proposition~\ref{prop:dp/p=mu},
  \[
    \sum_{n\ge0}\mu_n^{\le m}(b,a,\lambda) x^n = \frac{\delta p^*_{m}(x)}{p^*_{m+1}(x)} =
   (1-q^{-\nu-1})x^{-1/2} \frac{R_{m,\nu+2}(x^{1/2};q^{-1})}{R_{m+1,\nu+1}(x^{1/2};q^{-1})},
  \]
  which gives the first identity with $z=x^{1/2}$. By taking the limit and using
  \eqref{eq:R=J2} we obtain the second identity.
\end{proof}

Now we review Flajolet's theory \cite{Flajolet1980} on continued fraction
expressions for Motzkin path generating functions.

\begin{defn}
A \emph{Motzkin path} is a lattice path from $(0,0)$ to $(n,0)$ consisting of up
steps $(1,1)$, down steps $(1,-1)$, and horizontal steps $(1,0)$ that never goes
below the $x$-axis. A \emph{2-Motzkin path} is a Motzkin path in which
every horizontal step is colored red or blue. The \emph{height} of a 2-Motzkin
path is the largest integer $y$ for which $(x,y)$ is a point in the path.

Denote by $\Mot_2$ the set of all 2-Motzkin paths and by
$\Mot^{\le m}_2$ the set of all 2-Motzkin paths with height at most $m$.

For sequences $\{a_n\}, \{b_n\}, \{c_n\}$, and $\{d_n\}$, define the
\emph{weight} $\wt(p;a,b,c,d)$ of a 2-Motzkin path $p$ to be the product of
$a_n$ (resp.~$b_n$, $c_{n}$, and $d_{n}$) for each red horizontal step
(resp.~blue horizontal step, up step, and down step) starting at height $n$, see
Figure~\ref{fig:Motz}.
\end{defn}

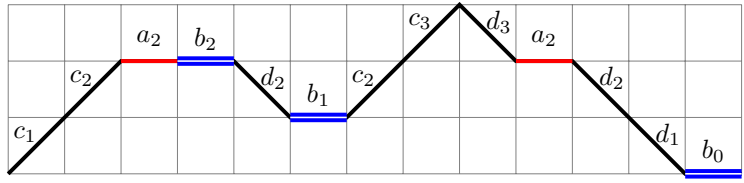
\begin{figure}
  \centering
\begin{tikzpicture}[scale=0.75]
    \draw[help lines] (0,0) grid (13,3);
    \draw[line width = 1.5pt] (0,0)-- ++(1,1)-- ++(1,1)-- ++(1,0)-- ++(1,0)-- ++(1,-1)-- ++(1,0)--
    ++(1,1)-- ++(1,1)-- ++(1,-1)-- ++(1,0)-- ++(1,-1)-- ++(1,-1)-- ++(1,0);
    \node at (.3,.7) {$c_{1}$};
    \node at (1.3,1.7) {$c_{2}$};
    \node at (6.3,1.7) {$c_{2}$};
    \node at (7.3,2.7) {$c_{3}$};
    \node at (2.5,2.4) {$a_{2}$};
    \node at (3.5,2.4) {$b_{2}$};
    \node at (4.7,1.7) {$d_{2}$};
    \node at (5.5,1.4) {$b_{1}$};
    \node at (8.7,2.7) {$d_{3}$};
    \node at (9.5,2.4) {$a_{2}$};
    \node at (10.7,1.7) {$d_{2}$};
    \node at (11.7,0.7) {$d_{1}$};
    \node at (12.5,0.4) {$b_{0}$};
    \draw[line width = 1.5pt, red] (2,2)--(3,2);
    \draw[line width = 1.5pt, blue,double] (3,2)--(4,2);
    \draw[line width = 1.5pt, blue,double] (5,1)--(6,1);
    \draw[line width = 1.5pt, red] (9,2)--(10,2);
    \draw[line width = 1.5pt, blue,double] (12,0)--(13,0);
\end{tikzpicture}
\caption{A 2-Motzkin path $p$ in $\Mot_2^{\le 3}$ with $\wt(p;a,b,c,d)=a_2^2
  b_0b_1b_2 c_1c_2^2c_3d_1d_2^2d_3$. The blue horizontal edges are represented by double
  edges.}
  \label{fig:Motz}
\end{figure}

Flajolet's theory \cite{Flajolet1980} proves the following lemma.

\begin{lem}\label{lem:flajolet}
  Given sequences $a,b,c,d$, we have
\[
\sum_{p\in\Mot_2^{\le m}} \wt(p;a,b,c,d) = \cfrac{1}{
  1-a_0-b_0 -\cfrac{c_0d_1}{
    1-a_1-b_1-
    \genfrac{}{}{0pt}{1}{}{\ddots \displaystyle - \cfrac{c_{m-1}d_m}{1-a_m-b_m}}
}}.
\]
\end{lem}

Observe that a connected skew shape $\alpha\in\CS$ is determined by the two
boundary paths starting from the bottom-left point to the top-right point
consisting of north steps and east steps, which form the boundary of $\alpha$.
Let $U(\alpha)$ be the upper boundary path and $D(\alpha)$ the lower boundary
path, see Figure~\ref{fig:UD paths}.

\begin{figure}
  \centering
\begin{tikzpicture}[scale=0.5]
    \draw[help lines] (0,0) grid (8,7);
    \draw[line width = 1.5pt] (0,0)--(0,4)--(3,4)--(3,6)--(4,6)--(4,7)--(8,7);
    \draw[line width = 1.5pt,dashed] (0,0)--(3,0)--(3,1)--(4,1)--(4,2)--(7,2)--(7,6)
    --(8,6)--(8,7);
    \node[fill=white] at (1.5,5) {$U(\alpha)$};
    \node[fill=white] at (5.5,1) {$D(\alpha)$};
\end{tikzpicture}
\caption{The boundary paths $U(\alpha)$ and $D(\alpha)$ for the connected skew
  shape $\alpha=(8,7,7,7,4,3)/(4,3,3)$.}
  \label{fig:UD paths}
\end{figure}
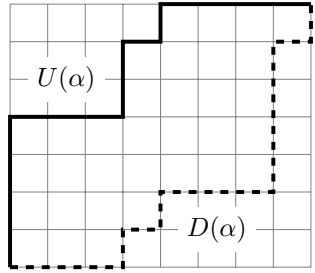

There is a well known bijection between 2-Motzkin paths and connected skew
shapes. 

\begin{defn} [The map $\phi:\Mot^{\le m}_2 \to \CS^{\le m}$]
 Let $p\in \Mot_2$.  
  Then $\phi(p)=\alpha$ is the connected skew shape whose upper and lower
  boundary paths $u,d$ are constructed by the following algorithm.
  \begin{enumerate}
  \item The first step of $u$ (resp.~$d$) is a north (resp.~east) step.
  \item For $i=1,2,\dots,n$, where $n$ is the number of steps in $p$,
    the $(i+1)$st steps of $u$ and $d$ are defined as follows.
    \begin{enumerate}
    \item If the $i$th step of $p$ is an up step, then
      the $(i+1)$st step of $u$ (resp.~$d$) is a north (resp.~east) step.
    \item If the $i$th step of $p$ is a down step, then
      the $(i+1)$st step of $u$ (resp.~$d$) is a east (resp.~north) step.
    \item If the $i$th step of $p$ is a red horizontal step, then
      the $(i+1)$st steps of $u$ and $d$ are both north steps.
    \item If the $i$th step of $p$ is a blue horizontal step, then
      the $(i+1)$st steps of $u$ and $d$ are both east steps.
    \end{enumerate}
  \item Finally, the last (the $(n+2)$nd) step of $u$ (resp.~$d$) is an east
    (resp.~north) step.
  \end{enumerate}
\end{defn}

For example, if $p$ is the $2$-Motzkin path in Figure~\ref{fig:Motz}, then
$\phi(p)$ is the connected skew shape $\alpha$ in Figure~\ref{fig:UD paths}.

\begin{prop}\label{prop:bijection}
  The map $\phi:\Mot^{\le m}_2 \to \CS^{\le m}$ is a bijection. Moreover, if
  $a,b,c$, and $d$ are the sequences given by $a_n=q^{n+1}y$, $b_n=q^{n+1}x$,
  $c_n=q^{n+1}xy$, and $d_n=q^{n+1}$, and if $\phi(p)=\alpha$, then
\[
x^{\col(\alpha)} y^{\row(\alpha)} q^{\area(\alpha)} = qxy\cdot \wt(p;a,b,c,d).
\]
\end{prop}
\begin{proof}
  One can easily prove this by the observation that if the $i$th step of $p$
  starts at height $n$, then the number of squares whose centers are on the line
  connecting the starting points of the $(i+1)$st steps of $U(\alpha)$ and
  $D(\alpha)$ is $n+1$, see Figure~\ref{fig:area}. We omit the details.
\end{proof}

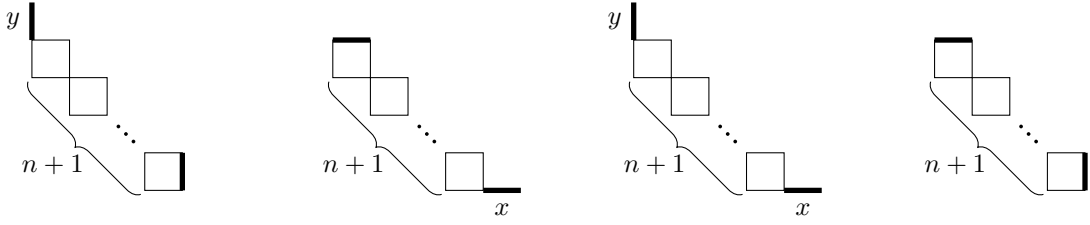
\begin{figure}
  \centering
\begin{tikzpicture}[scale=0.5]
  \draw (0,0) rectangle (1,1);
  \draw (1,-1) rectangle (2,0);
  \node at (2.3,-1.3) [circle,fill,inner sep=0.5pt]{};
  \node at (2.5,-1.5) [circle,fill,inner sep=0.5pt]{};
  \node at (2.7,-1.7) [circle,fill,inner sep=0.5pt]{};
  \draw (3,-3) rectangle (4,-2);
  \draw [decorate,decoration={brace,amplitude=5pt},xshift=-3pt,yshift=-3pt]
  (3,-3) -- (0,0) node [black,midway,xshift=-12pt, yshift=-10pt]{$n+1$};
  \draw[line width = 2pt] (0,1)--(0,2);
  \draw[line width = 2pt] (4,-3)--(4,-2);
  \node at (-.5,1.5) {$y$};

  \begin{scope}[shift={(8,0)}]
  \draw (0,0) rectangle (1,1);
  \draw (1,-1) rectangle (2,0);
  \node at (2.3,-1.3) [circle,fill,inner sep=0.5pt]{};
  \node at (2.5,-1.5) [circle,fill,inner sep=0.5pt]{};
  \node at (2.7,-1.7) [circle,fill,inner sep=0.5pt]{};
  \draw (3,-3) rectangle (4,-2);
  \draw [decorate,decoration={brace,amplitude=5pt},xshift=-3pt,yshift=-3pt]
  (3,-3) -- (0,0) node [black,midway,xshift=-12pt, yshift=-10pt]{$n+1$};
  \draw[line width = 2pt] (0,1)--(1,1);
  \draw[line width = 2pt] (4,-3)--(5,-3);
  \node at (4.5,-3.5) {$x$};
  \end{scope}
  
  \begin{scope}[shift={(16,0)}]
  \draw (0,0) rectangle (1,1);
  \draw (1,-1) rectangle (2,0);
  \node at (2.3,-1.3) [circle,fill,inner sep=0.5pt]{};
  \node at (2.5,-1.5) [circle,fill,inner sep=0.5pt]{};
  \node at (2.7,-1.7) [circle,fill,inner sep=0.5pt]{};
  \draw (3,-3) rectangle (4,-2);
  \draw [decorate,decoration={brace,amplitude=5pt},xshift=-3pt,yshift=-3pt]
  (3,-3) -- (0,0) node [black,midway,xshift=-12pt, yshift=-10pt]{$n+1$};
  \draw[line width = 2pt] (0,1)--(0,2);
  \draw[line width = 2pt] (4,-3)--(5,-3);
  \node at (4.5,-3.5) {$x$};
  \node at (-.5,1.5) {$y$};
  \end{scope}

  \begin{scope}[shift={(24,0)}]
  \draw (0,0) rectangle (1,1);
  \draw (1,-1) rectangle (2,0);
  \node at (2.3,-1.3) [circle,fill,inner sep=0.5pt]{};
  \node at (2.5,-1.5) [circle,fill,inner sep=0.5pt]{};
  \node at (2.7,-1.7) [circle,fill,inner sep=0.5pt]{};
  \draw (3,-3) rectangle (4,-2);
  \draw [decorate,decoration={brace,amplitude=5pt},xshift=-3pt,yshift=-3pt]
  (3,-3) -- (0,0) node [black,midway,xshift=-12pt, yshift=-10pt]{$n+1$};
  \draw[line width = 2pt] (4,-3)--(4,-2);
  \draw[line width = 2pt] (0,1)--(1,1);
  \end{scope}
\end{tikzpicture}
\caption{From left to right are shown the pairs of steps in $U(\alpha)$ and
  $D(\alpha)$ corresponding to a red horizontal step, a blue horizontal step, an
  up step, and a down step starting at height $n$ in $p$ when $\phi(p)=\alpha$.
  The number of squares whose centers are on the line connecting the starting
  points of the steps in $U(\alpha)$ and $D(\alpha)$ is $n+1$. In each diagram
  an $x$ (resp.~$y$) is written when a new column (resp.~row) is created.}
  \label{fig:area}
\end{figure}

Proposition~\ref{prop:CSm CF} in the introduction follows from
Lemma~\ref{lem:flajolet} and Proposition~\ref{prop:bijection}. Now we can easily
prove Theorem~\ref{thm:main1} in the introduction. Let us state the theorem
again.

\begin{thm}
  For any integer $m\ge 1$, we have
\[
  \sum_{\alpha\in \CS^{\le m}} (q^\nu z^2)^{\col(\alpha)} (q^\nu)^{\row(\alpha)} q^{\area(\alpha)}=
-q^\nu z \frac{R_{m,\nu+2}(z;q^{-1})}{R_{m+1,\nu+1}(z;q^{-1})}.
\] 
\end{thm}
\begin{proof}
  Let $x=q^\nu z^2$ and $y=q^\nu$. Then by Proposition~\ref{prop:CSm CF} the
  left hand side of the equation is
\begin{equation}\label{eq:xyq}
  \sum_{\alpha\in \CS^{\le m}} x^{\col(\alpha)} y^{\row(\alpha)} q^{\area(\alpha)}= \cfrac{qxy}{
  1-q(x+y) -\cfrac{q^3xy}{
    1-q^2(x+y)-
    \genfrac{}{}{0pt}{1}{}{\displaystyle\ddots - \cfrac{q^{2m+1}xy}{1-q^{m+1}(x+y)}}
}}.
\end{equation}
Using Theorem~\ref{thm:ratio of R-1}, Definition~\ref{defn:mu}, and
Lemma~\ref{lem:K=K} with $c_i=1-q^{\nu+i+1}$, we obtain that the right hand side
of the equation is 
\begin{align*}
  -q^\nu z \frac{R_{m,\nu+2}(z;q^{-1})}{R_{m+1,\nu+1}(z;q^{-1})} 
  &= \frac{q^{2\nu+1}z^2}{1-q^{\nu+1}} \sum_{n\ge0}\mu_n^{\le m}(b,a,\lambda) z^{2n}\\
  &=  -(1-q^{\nu}) \KK_{i=0}^m
    \left( \frac{-q^{2\nu+2i+1}z^2/(1-q^{\nu+i})(1-q^{\nu+i+1})}
    {1-q^{\nu+i+1}z^2/(1-q^{\nu+i+1})} \right)\\
  &=  - \KK_{i=0}^m \left( \frac{-q^{2\nu+2i+1}z^2}
    {1-q^{\nu+i+1}-q^{\nu+i+1}z^2} \right)
  = - \KK_{i=0}^m \left( \frac{-q^{2i+1}xy} {1-q^{i+1}(x+y)} \right),
\end{align*}
which is the same as the right hand side of \eqref{eq:xyq}. 
\end{proof}

\section{$q$-N\"orlund and Heine continued fractions}
\label{sec:q-norlund-heine}

In the previous section we have shown that the ratio
$J_{\nu+1}(z;q^{-1})/J_{\nu}(z;q^{-1})$ is a generating function for moments of
orthogonal polynomials of type $R_I$. In this section we show that this ratio
can also be written a generating function for moments of usual orthogonal
polynomials.

Recall from \eqref{eq:J/J} that
\[
  \frac{J_{\nu+1}(z;q^{-1})}{J_{\nu}(z;q^{-1})}= \frac{-q^{\nu+1}z}{1-q^{\nu+1}}\cdot
  \frac{\qhyper11{0}{q^{\nu+2}}{q,q^{\nu+2}z^2}}
  {\qhyper11{0}{q^{\nu+1}}{q,q^{\nu+1}z^2}}.
\]
Our strategy is to find two continued fraction expressions for the ratio of
${}_1\phi_1$'s in the above equation using the $q$-N\"orlund continued fraction
and Heine's continued fraction. 

First we state the $q$-N\"orlund fraction \cite[(19.2.7)]{Cuyt}.

\begin{lem}[$q$-N\"orlund fraction]
  \label{lem:norlund}
  We have
\[
\frac{\qhyper21{a,b}c{q,z}}{\qhyper21{aq,bq}{cq}{q,z}}= 
\frac{1-c-(a+b-ab-abq)z}{1-c}+\frac{1}{1-c} \KK_{m=1}^\infty
\left( \frac{c_m(z)}{e_m+d_mz} \right),
\]
where
\begin{align*}
  c_m(z)&=(1-aq^m)(1-bq^m)(cz-abq^mz^2)q^{m-1},\\
  e_m&=1-cq^m,\\
  d_m&=-(a+b-abq^m-abq^{m+1})q^m.
\end{align*}
\end{lem}

The $q$-N\"orlund fraction can be restated in the form of a continued fraction
for type $R_I$ orthogonal polynomials.

\begin{prop} [$q$-N\"orlund fraction restated]
  \label{prop:q-Norlund}
  We have
  \[
\frac{\qhyper21{aq,bq}{cq}{q,z}}{\qhyper21{a,b}c{q,z}}
  = \cfrac{1}{
    1-b_0z -\cfrac{a_1z+\lambda_1 z^2}{
      1-b_1z-\cfrac{a_2z+\lambda_2 z^2}{
        1-b_2z- \genfrac{}{}{0pt}{1}{}{\ddots}}}},
  \]
  where
  \begin{align*}
b_m & = \frac{(a+b-abq^m-abq^{m+1})q^m}{1-cq^m},\\
a_m & = -\frac{(1-aq^m)(1-bq^m)cq^{m-1}}{(1-cq^{m-1})(1-cq^m)},\\
\lambda_m & = \frac{(1-aq^m)(1-bq^m)abq^{2m-1}}{(1-cq^{m-1})(1-cq^m)}.
  \end{align*}
\end{prop}
\begin{proof}
  By taking the inverse on each side of the equation in Lemma~\ref{lem:norlund}
  we obtain
\begin{equation}
  \label{eq:bm}
\frac{\qhyper21{aq,bq}{cq}{q,z}}{\qhyper21{a,b}c{q,z}}= 
\frac{1-c}{c_0(z)} \KK_{m=0}^\infty
\left( \frac{c_m(z)}{e_m+d_mz} \right).
\end{equation}
Applying Lemma~\ref{lem:K=K} with $c_i=1/(1-cq^i)$ and $m\to\infty$ yields
\[
\frac{\qhyper21{aq,bq}{cq}{q,z}}{\qhyper21{a,b}c{q,z}}= 
\frac{(1-cq^{-1})(1-c)}{c_0(z)} \KK_{m=0}^\infty
\left( \frac{c_m(z)/(1-cq^{m-1})(1-cq^m)}
  {e_m/(1-cq^m) + d_mz/(1-cq^m)} \right),
\]
which is the same as the desired identity.
\end{proof}

Heine's contiguous relation \cite[17.6.19]{NIST:DLMF} is 
\[
  \qhyper21{aq,b}{cq}{q,z} - \qhyper21{a,b}{c}{q,z}
  = \frac{(1-b)(a-c)z}{(1-c)(1-cq)} \qhyper21{aq,bq}{cq^2}{q,z}.
\]
Equivalently, 
\begin{equation}
  \label{eq:Heine_contigous2}
  \frac{\qhyper21{aq,b}{cq}{q,z}}{\qhyper21{a,b}{c}{q,z}}
  =\cfrac{1}{1- \cfrac{(1-b)(a-c)z}{(1-c)(1-cq)}
  \cdot \cfrac{\qhyper21{bq,aq}{cq^2}{q,z}}{\qhyper21{b,aq}{cq}{q,z}}}.
\end{equation}
Applying \eqref{eq:Heine_contigous2} iteratively, we obtain Heine's continued
fraction, which is a $q$-analog of Gauss's continued fraction.

\begin{lem}[Heine's fraction]
  \label{lem:Heine}
We have  
\[
  \frac{\qhyper21{aq,b}{cq}{q,z}}{\qhyper21{a,b}{c}{q,z}} =
  \cfrac{1}{1-\cfrac{\beta_1z}{1-\cfrac{\beta_2z}{1-\cdots}}},
\]
where
\[
  \beta_{2n+1} = \frac{(1-bq^n)(a-cq^n)q^{n}}{(1-cq^{2n})(1-cq^{2n+1})},
  \qquad
  \beta_{2n} = \frac{(1-aq^n)(b-cq^n)q^{n-1}}{(1-cq^{2n-1})(1-cq^{2n})}.
\]
\end{lem}

Now we give two different continued fraction expressions for a ratio of
${}_1\phi_1$'s.

\begin{prop}\label{prop:two cont}
  We have
\begin{align}
  \label{eq:rato 1phi1 ba}
\frac{\qhyper11{0}{cq}{q,qz}}{\qhyper11{0}{c}{q,z}} 
&= \cfrac{1}{1-b_0z - \cfrac{a_1z}{1-b_1z-\cfrac{a_2z}{1-b_2z-\cdots}}}\\
  \label{eq:rato 1phi1 la}
&=  \cfrac{1}{1-\cfrac{\lambda_1z}{1-\cfrac{\lambda_2z}{1-\cdots}}},
\end{align}
where
\[
a_i = \frac{cq^{2i-1}}{(1-cq^{i-1})(1-cq^{i})}, \qquad
b_i = \frac{q^{i}}{1-cq^{i}},
\]
and
\[
  \lambda_{2i} = \frac{cq^{3i-1}}{(1-cq^{2i-1})(1-cq^{2i})},
  \qquad
  \lambda_{2i+1} = \frac{q^{i}}{(1-cq^{2i})(1-cq^{2i+1})}.
\]
\end{prop}
\begin{proof}
We use the well known fact
\cite[p.5]{GR}:
\[
\lim_{a_1\to\infty}  \qhyper rs{a_1,\dots,a_r}{b_1,\dots,b_s}{q,\frac{z}{a_1}} =
  \qhyper {r-1}s{a_2,\dots,a_r}{b_1,\dots,b_s}{q,z}.
\]
Equation~\eqref{eq:rato 1phi1 ba} (resp.~Equation~\eqref{eq:rato 1phi1 la}) is
obtained by replacing $b\mapsto 0$, $z\mapsto z/a$ and sending $a$ to infinity
in Proposition~\ref{prop:q-Norlund} (resp.~Lemma~\ref{lem:Heine}).
\end{proof}

Using Proposition~\ref{prop:two cont} we obtain two different continued fraction
expressions for the ratio $J_{\nu+1}(z;q^{-1})/J_{\nu}(z;q^{-1})$, one of which
is given in Theorem~\ref{thm:ratio of R-1}.

\begin{thm}\label{thm:ratio of J-1}
  Let $b,a,\lambda,b',a',\lambda'$ be the sequences
  given by 
  \[
b_n=\frac{q^{\nu+n+1}}{1-q^{\nu+n+1}},\qquad a_n
  =\frac{q^{2\nu+2n+1}}{(1-q^{\nu+n})(1-q^{\nu+n+1})} ,\qquad \lambda_n =0,
  \]
  \[
b'_n=a'_n=0, \qquad
  \lambda'_{2i} = \frac{q^{2\nu+3i+1}}{(1-q^{\nu+2i})(1-q^{\nu+2i+1})},
  \qquad
  \lambda'_{2i+1} = \frac{q^{\nu+i+1}}{(1-q^{\nu+2i+1})(1-q^{\nu+2i+2})}.
  \]
  Then
  \begin{align}
    \label{eq:ratio of J-12}
    \frac{J_{\nu+1}(z;q^{-1})}{J_{\nu}(z;q^{-1})}
    &= \frac{q^{\nu+1}z}{q^{\nu+1}-1} \sum_{n\ge0}\mu_n(b,a,\lambda) z^{2n}\\
    \label{eq:ratio of J-1'}
   &= \frac{q^{\nu+1}z}{q^{\nu+1}-1} \sum_{n\ge0}\mu_{2n}(b',a',\lambda') z^{2n}.
  \end{align}
\end{thm}
\begin{proof}
  The first identity \eqref{eq:ratio of J-12} has been already proved in
  Theorem~\ref{thm:ratio of R-1}. Alternatively, applying \eqref{eq:rato
    1phi1 ba} to \eqref{eq:J/J} gives another proof of \eqref{eq:ratio of J-12}.
  Applying \eqref{eq:rato 1phi1 la} to \eqref{eq:J/J} gives the second identity
  \eqref{eq:ratio of J-1'}.
\end{proof}

Note that the above theorem shows that $J_{\nu+1}(z;q^{-1})/J_{\nu}(z;q^{-1})$
is (up to a scalar multiplication) the moment generating function for both
orthogonal polynomials of type $R_I$ and usual orthogonal polynomials. Using
Flajolet's theory on continued fractions one obtains two combinatorial
interpretations for the coefficients in the series expansion of this ratio. Note
also that we get a similar result for the ratio $J_{\nu+1}(z;q)/J_{\nu}(z;q)$ if
we replace $q$ by $q^{-1}$ in Theorem~\ref{thm:ratio of J-1}.

\begin{remark}
By replacing $a\mapsto 0$, $x\mapsto x/b$ and sending $b$ to infinity in
Lemma~\ref{lem:Heine} we obtain 
\begin{equation}
  \label{eq:1phi1}
  \frac{\qhyper11{0}{cq}{q,z}}{\qhyper11{0}{c}{q,z}} =
  \cfrac{1}{1-\cfrac{\lambda_1z}{1-\cfrac{\lambda_2z}{1-\cdots}}},
\end{equation}
where
\[
  \lambda_{2n} = \frac{q^{n-1}}{(1-cq^{2n-1})(1-cq^{2n})},
  \qquad
  \lambda_{2n+1} = \frac{cq^{3n}}{(1-cq^{2n})(1-cq^{2n+1})}.
\]
One can easily check that \eqref{eq:1phi1} is equivalent to
\eqref{eq:rato 1phi1 la}
using the fact
\[
\qhyper11{0}{c}{q,z} = \qhyper11{0}{c^{-1}}{q^{-1},q^{-1}c^{-1}z},
\]
which implies
\[
\frac{\qhyper11{0}{qc}{q,qz}}{\qhyper11{0}{c}{q,z}} =
\frac{\qhyper11{0}{q^{-1}c^{-1}}{q^{-1},q^{-1}c^{-1}z}}{\qhyper11{0}{c^{-1}}{q^{-1},q^{-1}c^{-1}z}}.
\]
\end{remark}

\begin{remark}
  Recall \eqref{eq:tan} that the ratio $J_{1/2}(z)/J_{-1/2}(z)$ is the tangent
  function. The \emph{tangent numbers} or \emph{(odd) Euler numbers} $E_{2n+1}$
  are defined by
\[
\tan x= \sum_{n=0}^\infty \frac{E_{2n+1}x^{2n+1}}{(2n+1)!}.
\]
The following $q$-analog of tangent numbers has been studied in several papers
\cite{Fulmek2000, Huber2010, Hwang2019,MPP2,Prodinger2008}:
\[
\sum_{n=0}^\infty \frac{E^*_{2n+1}(q)x^{2n+1}}{(q;q)_{2n+1}}
=\frac{x}{1-q} \cdot
\frac{\qhyper11{0}{q^3}{q^2,q^2x^2}}{\qhyper11{0}{q}{q^2,x^2}}
= -q^{1/2} \frac{J_{1/2}(q^{-1/2}x;q^{-2})}{J_{-1/2}(q^{-1/2}x;q^{-2})},
\]
where the last equality follows from \eqref{eq:J/J}. Hwang et al.
\cite[(30),(31)]{Hwang2019} found two continued fractions for this generating
function, which are both special cases of Theorem~\ref{thm:ratio of J-1}. There
are combinatorial objects related to this generating function such as
alternating permutations and skew semistandard Young tableaux \cite{Hwang2019,
  MPP2, Stanley2010}. It would be interesting to generalize these results for an
arbitrary $\nu$.
\end{remark}

The ratio of two ${}_1\phi_1$'s in Proposition~\ref{prop:two cont} is related to
$J_{\nu+1}(z;q^{-1})/J_\nu(z;q^{-1})$. Similarly, we can consider 
the ratio corresponding to
\begin{equation}\label{eq:2}
   \frac{J_{\nu+1}^{(1)}(z;q)}{J_\nu^{(1)}(z;q)} =
   \frac{z/2}{1-q^{\nu+1}}
   \frac{\qhyper21{0,0}{q^{\nu+2}}{q,-z^2/4}}{\qhyper21{0,0}{q^{\nu+1}}{q,-z^2/4}},
\end{equation}
where $J_\nu^{(1)}(z;q)$ is Jackson's $q$-Bessel function defined in
\eqref{eq:J1}. In this case, both Heine's fraction (Lemma~\ref{lem:Heine}) and
the $q$-N\"orlund fraction (Proposition~\ref{prop:q-Norlund}) with $a=b=0$ give
the same continued fraction:
\begin{equation}
  \label{eq:Heine0}
  \frac{\qhyper21{0,0}{cq}{q,z}}{\qhyper21{0,0}{c}{q,z}} =
  \cfrac{1}{1-\cfrac{a_1z}{1-\cfrac{a_2z}{1-\cdots}}},
\end{equation}
where
\[
a_{n} = \frac{-cq^{n-1}}{(1-cq^{n-1})(1-cq^{n})}.
\]
Applying \eqref{eq:Heine0} to \eqref{eq:2} we obtain the following theorem.

\begin{thm}\label{thm:ratio of J-12}
  Let $b,a,\lambda$ be the sequences
  given by 
  \[
b_n=\lambda_n=0, \qquad
  a_{n} = \frac{q^{\nu+n}}{(1-q^{\nu+n})(1-q^{\nu+n+1})}.
  \]
  Then
  \[
    \frac{J^{(1)}_{\nu+1}(z;q)}{J^{(1)}_{\nu}(z;q)}
    = \frac{z/2}{1-q^{\nu+1}} \sum_{n\ge0}\mu_{n}(b,a,\lambda) \left(- \frac{z^2}{4} \right)^{n}.
  \]
Equivalently,  letting $b',a',\lambda'$ be the sequences
  given by 
  \[
b'_n=a'_n=0, \qquad
  \lambda'_{n} = \frac{q^{\nu+n}}{(1-q^{\nu+n})(1-q^{\nu+n+1})},
  \]
we have
  \[
    \frac{J^{(1)}_{\nu+1}(z;q)}{J^{(1)}_{\nu}(z;q)}
    = \frac{z/2}{1-q^{\nu+1}} \sum_{n\ge0}\mu_{n}(b',a',\lambda') \left(\frac{iz}{2} \right)^{n}.
  \]
\end{thm}

Note that as before the above theorem shows that
$J^{(1)}_{\nu+1}(z;q)/J^{(1)}_{\nu}(z;q)$ is (up to a scalar multiplication) the
moment generating function for orthogonal polynomials, which gives a
combinatorial interpretation for the coefficients in the series expansion of
this ratio.

\section{Open problems}
\label{sec:further-study}

Recall that Kishore's theorem is a statement about the power
series coefficients of the ratio $J_{\nu+1}(x)/J_\nu(x)$ of two Bessel
functions. We conjecture the following finite version of Kishore's theorem on a
ratio of Lommel polynomials $R_{m,\nu}(x)$ defined in the introduction.

\begin{conj}\label{conj:kishore}
Let 
\[
\frac{R_{m,\nu+2}(x)}{R_{m+1,\nu+1}(x)}
= \sum_{n=0}^{\infty} \frac{N^{(m)}_{n,\nu}}{D^{(m)}_{n,\nu}} \left( \frac{x}{2} \right)^{2n+1},
\]
where 
\[
D^{(m)}_{n,\nu} = \prod_{k=0}^m (\nu+k+1)^{f(m,n,k)},
\]
\[
  f(m,n,k)=
  \begin{cases}
  \max\left(\displaystyle \flr{\frac{n+1}{k+1}},\flr{\frac{n+m-2k+1}{m-k+1}} \right),
  & \mbox{if $k\ne m/2$},\\  
  1,
  & \mbox{if $k=m/2$}.
  \end{cases}
\]
Then 
$N^{(m)}_{n,\nu}$ is a polynomial in $\nu$ with nonnegative integer coefficients.
\end{conj}

By \eqref{eq:hurwitz2} Conjecture~\ref{conj:kishore} implies Kishore's Theorem.

\medskip

In Section~\ref{sec:q-norlund-heine} we saw that the ratio
\[
  \frac{J_{\nu+1}(z;q^{-1})}{J_{\nu}(z;q^{-1})}= \frac{-q^{\nu+1}z}{1-q^{\nu+1}}\cdot
  \frac{\qhyper11{0}{q^{\nu+2}}{q,q^{\nu+2}z^2}}
  {\qhyper11{0}{q^{\nu+1}}{q,q^{\nu+1}z^2}}
\]
has two generalizations, the $q$-N\"orlund continued fraction and Heine's
continued fraction.
These two generalizations seem to have a similar property as follows.

\begin{conj}
  Let
  \[
     \sum_{n\ge0} \gamma_n(a,b,c) z^n =  \frac{\qhyper21{aq,bq}{cq}{q,z}}{\qhyper21{a,b}{c}{q,z}}.
    \]
Then    
  \[
    \frac{\gamma_{n}(a,b,c)}{1-c} = \frac{P_n(a,b,c)}{\prod_{k=0}^n(1-cq^k)^{\flr{\frac{n+1}{k+1}}}},
  \]
  for some polynomial $P_n(a,b,c)$ in $a,b,c,q$ with integer coefficients. 
\end{conj}

\begin{conj}
  Let
  \[
     \sum_{n\ge0} \gamma'_n(a,b,c) z^n =  \frac{\qhyper21{aq,b}{cq}{q,z}}{\qhyper21{a,b}{c}{q,z}}.
    \]
Then    
  \[
    \frac{\gamma'_{n}(a,b,c)}{1-c} = \frac{P'_n(a,b,c)}{\prod_{k=0}^n(1-cq^k)^{\flr{\frac{n+1}{k+1}}}},
  \]
  for some polynomial $P'_n(a,b,c)$ in $a,b,c,q$ with integer coefficients. 
\end{conj}

Using Flajolet's theory one can reinterpret the equality of the two continued
fractions in Proposition~\ref{prop:two cont} completely combinatorially.

\begin{problem}
  Find a combinatorial proof of the identity
\[
 \cfrac{1}{1-b_0z - \cfrac{a_1z}{1-b_1z-\cfrac{a_2z}{1-b_2z-\cdots}}}
=  \cfrac{1}{1-\cfrac{\lambda_1z}{1-\cfrac{\lambda_2z}{1-\cdots}}},
\]
where
\[
a_i = \frac{cq^{2i-1}}{(1-cq^{i-1})(1-cq^{i})}, \qquad
b_i = \frac{q^{i}}{1-cq^{i}},
\]
and
\[
  \lambda_{2i} = \frac{cq^{3i-1}}{(1-cq^{2i-1})(1-cq^{2i})},
  \qquad
  \lambda_{2i+1} = \frac{q^{i}}{(1-cq^{2i})(1-cq^{2i+1})}.
\]
\end{problem}

Finally we remark that more general ratios of hypergeometric series are
considered in \cite{Sokal_cfrac}. It would be interesting to see whether the
results in this paper can be extended to these ratios.

\end{document}